\documentclass{svjour3}                     
\smartqed  
\usepackage{graphicx}
\usepackage{amssymb,amsmath,amsfonts,latexsym,euscript}
\def\d{\,{\rm d}}
\def\E{\,{\rm e}}

\begin{document}

\title{Optimal interpolation formulas in $W_2^{(m,m-1)}$ space 
}

\titlerunning{Optimal interpolation formulas}        

\author{S.S. Babaev, A.R. Hayotov}


\institute{S.S. Babaev\at
Bukhara State University, 11, M.Iqbol str.,  Bukhara, Uzbekistan\\
\email{b\_samandar@mail.ru}
\and A.R. Hayotov
\at
V.I.Romanovskiy Institute of Mathematics, Uzbekistan Academy of Sciences, M.Ulugbek str. 81,
Tashkent 100125, Uzbekistan\\
\email{hayotov@mail.ru}
}

\date{Received: date / Accepted: date}

\maketitle
\begin{abstract}
In the present paper optimal
interpolation formulas are constructed in $W_2^{(m,m-1)}(0,1)$
space. Explicit formulas for coefficients of optimal interpolation
formulas are obtained. Some numerical results are presented.

\textbf{MSC:} 41A05, 41A15, 65D30, 65D32.

\textbf{Keywords:} optimal interpolation formula, optimal
coefficients, the error functional, the extremal function.
\end{abstract}

\maketitle


\
\section{Introduction: statement of the Problem}

The typical problem of approximation is the interpolation problem. The classical method of its solution consists of construction
of interpolation polynomials. However polynomials have some drawbacks as an instrument of approximation of functions with specifics and functions
with small smoothness. It is proved that the sequence of the Lagrange polynomials,
which constructed for concrete continues function for equally  spaced nodes, by increasing of degree is not converge to the function.
That is why splines are used instead of interpolation polynomials of high degree.

There are algebraic and variational approaches in the spline theory. In algebraic approach splines are considered as
some smooth piecewise polynomial functions. In variational approach splines are elements of
Hilbert or Banach spaces minimizing certain functionals. The first spline functions were constructed from pieces of cubic polynomials.
After that, this construction was modified, the degree of polynomials increased, but the idea of their constructions remains permanently.
The next essential step in the theory of splines was Hollyday's result \cite{Hol57}, connecting Schoenberg's cubic splines with the solution of the variational problem on minimum of square of a function norm from the space $L_2^{(2)}$.   Further, the Hallyday result was generalized by Carl de Boor \cite{deBoor63}.
Further, a large number of papers appeared, where, depending on specific requirements, the variational functional was modified.
The theory of splines based on variational methods studied and developed, for example, by Ahlberg et al \cite{Ahlb67}, Arcangeli et al
\cite{Arc04}, Attea \cite{Attea92}, A.V.Bezhaev and  V.A.Vasilenko \cite{BezhVas01}, C. de Boor \cite{deBoor78}, M.I.Ignatev and
A.B.Pevniy \cite{Ign91}, P.J.Laurent \cite{Lor75}, I.Schoenberg \cite{Schoen64},  L.L.Schumaker \cite{Schum07}, S.T.Stechkin and Yu.N.Subbotin \cite{Stech76} and others.

The present paper is devoted to a variational method. Here we construct optimal interpolation formulas.

Assume we are given the table of the values $\varphi(x_\beta)$, $\beta=0,1,...,N$ of functions $\varphi$ at points $x_{\beta}\in [0,1]$.
It is required approximate functions $\varphi$ by another more simple function $P_{\varphi}$, i.e.
\begin{equation} \label{(1)}
\varphi (x) \cong P_\varphi  (x) = \sum\limits_{\beta = 0}^N
{C_\beta(x) \cdot \varphi (x_\beta )},
\end{equation}
which satisfies the following equalities
\begin{equation}\label{(2)}
 \varphi(x_{\beta})=P_{\varphi}(x_{\beta}), \ \beta=0,1,...,N.
\end{equation}
Here $C_\beta (x)$ and $x_\beta$ ($\in [0,1]$) are
\emph{the coefficients} and \emph{the nodes} of the interpolation formula
(\ref{(1)}), respectively.

We suppose that functions
$\varphi$ belong to the Hilbert space
$$
W_2^{(m,m-1)}(0,1)=\{\varphi:[0,1]\to \mathbb{R}\ |\
\varphi^{(m-1)} \mbox{ is abs. cont. and }
\varphi^{(m)}\in L_2(0,1)\},
$$
equipped with the norm
\begin{equation}\label{(3)}
\left\| {\varphi|W_2^{(m,m-1)}(0,1)} \right\| = \left\{
{\int\limits_0^1 {\left( \varphi^{(m)}(x)+\varphi^{(m-1)}(x)
\right)} ^2 \d x} \right\}^{1/2}
\end{equation}
and ${\int\limits_0^1 {\left( \varphi^{(m)}(x)+\varphi^{(m-1)}(x)
\right)} ^2 \d x}<\infty$. The equality (\ref{(3)}) is
the semi-norm and $\|\varphi\|=0$ if and only if
$\varphi(x)=c_0+c_1x+...+c_{m-2}x^{m-2}+c_{m-1}e^{-x}$.

It should be noted that for a linear differential operator of
order $m$, $L\equiv
a_m\frac{\d^m}{\d x^m}+a_{m-1}\frac{\d^{m-1}}{\d x^{m-1}}+...+a_1\frac{\d}{\d x}+a_0$, $a_m\neq 0$,
Ahlberg, Nilson and Walsh in the book \cite[Chapter 6]{Ahlb67}
investigated the Hilbert spaces in the context of generalized
splines. Namely, with the inner product
$$
\langle\varphi,\psi\rangle=\int_0^1L\varphi(x)\cdot L\psi(x) \d x.
$$

The difference $\varphi-P_{\varphi}$ is called \emph{the error} of
the interpolation formula (\ref{(1)}). The value of this
error at a point $z\in [0,1]$ is the linear functional on the space $W_2^{(m,m-1)}(0,1)$, i.e.
\begin{eqnarray}
(\ell,\varphi)&=&\varphi(z)-P_{\varphi}(z)=\varphi(z)-\sum\limits_{\beta=0}^N
C_\beta(z)\varphi(x_{\beta})\nonumber \\
&=&\int\limits_{-\infty}^{\infty}\left(\delta(x-z)-\sum\limits_{\beta=0}^N
C_\beta(z)\delta(x-x_\beta)\right)\varphi(x)\d x,  \label{(4)}
\end{eqnarray}
where $\delta(x)$ is the Dirac delta-function and
\begin{equation}\label{(5)}
\ell(x,z)=\delta(x-z)-\sum\limits_{\beta=0}^N
C_\beta(z)\delta(x-x_\beta)
\end{equation}
is \emph{the error functional} of the interpolation formula
(\ref{(1)}) and belongs to the space ${W_2^{(m,m-1)*}}(0,1)$.
The space ${W_2^{(m,m-1)*}}(0,1)$ is the conjugate to the
space\linebreak ${W_2^{(m,m-1)}}(0,1)$. Further, for convenience, we denote $\ell(x,z)$ by $\ell(x)$.

By the Cauchy-Schwarz inequality the absolute value of the error (\ref{(4)}) is estimated as follows
$$
|(\ell,\varphi)|\leq
\|\varphi|{W_2^{(m,m-1)}}\|\cdot
\|\ell|{W_2^{(m,m-1)*}}\|,
$$
where
$$
\left\| {\ell|{W_2^{(m,m-1)*}} } \right\| = \mathop {\sup }\limits_{\varphi,\ \|\varphi \| \neq 0}
\frac{|(\ell,\varphi)|}{\|\varphi\|}.
$$
Therefore, in order to estimate
the error of the interpolation formula (\ref{(1)}) on
functions of the space ${W_2^{(m,m-1)}}(0,1)$ it is required to
find the norm of the error functional $\ell$ in the
conjugate space ${W_2^{(m,m-1)*}}(0,1)$.

From here we get

\begin{problem}\label{Prob1}
Find the norm of the error functional $\ell$ of
the interpolation formula (\ref{(1)}) in the space
${W_2^{(m,m-1)*}}(0,1)$.
\end{problem}

It is clear that the norm of the error functional $\ell$ depends on
the coefficients $C_\beta(z)$ and the nodes $x_{\beta}$. The problem of minimization of the quantity $\|\ell\|$ by coefficients $C_{\beta}(z)$ is
the linear problem and by nodes $x_{\beta}$ is, in general, nonlinear and complicated problem.
We consider the problem of minimization of the quantity $\|\ell\|$ by coefficients $C_{\beta}(z)$ when nodes $x_{\beta}$ are fixed.

The coefficients $\mathring{C}_{\beta}(z)$ (if there exist) satisfying
the following equality
 \begin{equation}
 \label{(6)}
\left\| \mathring{\ell}|{W_2^{(m,m-1)*}}\right\| = \mathop {\inf }\limits_{C_\beta(z) }  \left\|\ell|{W_2^{(m,m-1)*}} \right\|
 \end{equation}
are called \emph{the optimal coefficients} and corresponding interpolation formula
$$
\mathring{P}_{\varphi}(z)=\sum\limits_{\beta=0}^N\mathring{C}_{\beta}(z)\varphi(x_\beta)
$$
is called \emph{the optimal interpolation formula} in the space $W_2^{(m,m-1)*}(0,1)$.

Thus, in order to construct the optimal interpolation formula in
the space ${W_2^{(m,m-1)}}(0,1)$ we need to solve the next
problem.

\begin{problem}\label{Prob2}
Find the coefficients $\mathring{C}_\beta(z)$ which satisfy equality
(\ref{(6)}) when the nodes $x_{\beta}$ are fixed.
\end{problem}

The main aim of the present paper is to construct the optimal interpolation formulas
in ${W_2^{(m,m-1)}}(0,1)$ space and to find explicit formulas for optimal coefficients. First such a problem
was stated and studied by S.L. Sobolev in \cite{Sob61a},
where the extremal function of the interpolation formula was found
in the Sobolev space $W_2^{(m)}$.

The rest of the paper is organized as follows. In Section 2 the
extremal function which corresponds to the error functional
$\ell$ is found and with its help representation of the norm of
the error functional (\ref{(5)}) is calculated, i.e. Problem
1 is solved; in Section 3 in order to find the minimum of
$\left\| \ell \right\|^2 $ by coefficients $C_\beta(z)$ the system
of linear equations is obtained for the coefficients of optimal
interpolation formulas (\ref{(1)}) in the space
$W_2^{(m,m-1)}(0,1)$, moreover existence and uniqueness of the
solution of this system are proved; in Section 4 we give the algorithm for finding of coefficients of
optimal interpolation formulas (\ref{(1)}); Section 5 is
devoted to calculation of optimal coefficients using the algorithm
which is given in Section 4; finally, in Section 6 we give some
numerical results which confirm the theoretical results of the present
paper.

\section{The extremal function and the norm of the error functional $\ell$ }

In this section we solve  Problem 1, i.e. we find explicit form of
the norm of the error functional $\ell$.

For finding the explicit form of the norm of the error functional
$\ell$ in the space ${W_2^{(m,m-1)}}$ we use its extremal function which was introduced by Sobolev
\cite{Sob61a,Sob74}. The function  $\psi_{\ell}$ from
${W_2^{(m,m-1)}}(0,1)$ space is called \emph{the extremal
function} for the error functional $\ell$ if the following
equality is fulfilled
$$
\left( {\ell,\psi_{\ell}} \right) = \left\| {\ell\left|
{{W_2^{(m,m-1)*} }} \right.} \right\| \cdot \left\|
{\psi_{\ell}\left|{{W_2^{(m,m-1)}}} \right.} \right\|.
$$
The space ${W_2^{(m,m-1)}}(0,1)$ is a Hilbert space and the
inner product in this space is defined by the following formula
\begin{equation}\label{(7)}
\langle {\varphi,\psi} \rangle  =
\int\limits_0^1 \left(\varphi ^{(m)} (x)+\varphi ^{(m-1)}
(x)\right)\left(\psi ^{(m)}(x)+\psi ^{(m-1)}(x)\right)\d x.
\end{equation}
According to the Riesz theorem any linear continuous functional
$\ell$ in a Hilbert space is represented in the form of a
inner product. So, in our case, for any function $\varphi$ from ${W_2^{(m,m-1)}}(0,1)$
space, we have
\begin{equation} \label{(8)}
\left( {\ell,\varphi} \right) = \langle{\psi _\ell,\varphi} \rangle
\end{equation}
Here $\psi_\ell$ is the function from
${W_2^{(m,m-1)}}(0,1)$ is defined  uniquely by functional $\ell$ and is the extremal function.

It is easy to see from (\ref{(8)}) that the error functional $\ell$, defined on the space $W_2^{(m,m-1)}(0,1)$, satisfies the following equalities
\begin{eqnarray}
&&(\ell,x^{\alpha})=0,\quad\alpha=0,1,...,m-2,\label{(9)}\\
&&(\ell,\E^{-x})=0.    \label{(10)}
\end{eqnarray}
The equalities (\ref{(9)}) and (\ref{(10)}) mean that our interpolation formula is exact for the function $\E^{-x}$ and
for any polynomial of degree $\leq m-2$.

The equation (\ref{(8)}) was solved in Theorem 2.1 of \cite{ShadHay14} and for the extremal function $\psi_{\ell}$ was obtained the following
expression
\begin{equation}\label{(11)}
\psi_{\ell}(x)=(-1)^m\ell(x)*G_m(x)+P_{m-2}(x)+d\E^{-x},
\end{equation}
where
\begin{equation}\label{(12)}
G_m(x)=\frac{\mathrm{sgn}x}{2}\left(\frac{\E^x-\E^{-x}}{2}-\sum\limits_{k=1}^{m-1}\frac{x^{2k-1}}{(2k-1)!}\right),
\end{equation}
$P_{m-2}(x)$ is a polynomial of degree $m-2$, $d$ is a constant
and $*$ is the operation of convolution which for the functions $f$ and $g$ is defined as follows
$$
f(x)*g(x)=\int\limits_{-\infty}^{\infty}f(x-y)g(y)\d y=\int\limits_{-\infty}^{\infty}f(y)g(x-y)\d y.
$$

Now we obtain the norm of the error functional
$\ell$. Since the space ${W_2^{(m,m-1)}}(0,1)$ is the Hilbert
space then by the Riesz theorem we have
\begin{equation*}
 \left( {\ell,\psi_\ell} \right) = \|\ell\| \cdot \| \psi _\ell\| =\| \ell\|^2 .
\end{equation*}
Hence, using (\ref{(5)}) and (\ref{(11)}), taking
into account (\ref{(9)}) and (\ref{(10)}),  we get
{\small
\begin{eqnarray*}
\|\ell\|^2&=&
(\ell,\psi_{\ell})=\int\limits_{-\infty}^{\infty}\ell(x,z)\psi_{\ell}(x)\d x\\
&=& \int\limits_{-\infty}^{\infty}\ell(x,z) \bigg((-1)^m\bigg[
G_m(x-z)-\sum\limits_{\beta=0}^NC_{\beta}(z)G_m(x-x_{\beta})\bigg]+P_{m-2}(x)+d\E^{-x}
\bigg)\d x\\
&=&(-1)^m\int\limits_{-\infty}^{\infty}\left(\delta(x-z)-\sum\limits_{\beta=0}^NC_{\beta}(z)\delta(x-x_{\beta})\right)
\left(
G_m(x-z)-\sum\limits_{\beta=0}^NC_{\beta}(z)G_m(x-x_{\beta})\right)\d x.
\end{eqnarray*}
}
Hence, keeping in mind that $G_m(x)$, defined by (\ref{(12)}), is the even function, we
have
\begin{equation}
\|\ell\|^2=(-1)^m\left(\sum\limits_{\beta=0}^N\sum\limits_{\gamma=0}^N
C_{\beta}(z)C_{\gamma}(z)G_m(x_\beta-x_\gamma)-2\sum\limits_{\beta=0}^NC_{\beta}(z)G_m(z-x_{\beta})
\right).\label{(13)}
\end{equation}
Thus Problem \ref{Prob1} is solved.

Further, in next sections, we solve Problem \ref{Prob2}.

\section{Existence and uniqueness of optimal interpolation formula}

Assume that the nodes $x_\beta$  of the interpolation formula
(\ref{(1)}) are fixed. The error functional (\ref{(5)})
satisfies conditions (\ref{(9)}) and (\ref{(10)}). The
norm of the error functional $\ell$ is multidimensional
function with respect to the coefficients $C_\beta(z)$ $(\beta  =
\overline {0,N} )$. For finding the point of the conditional
minimum of the expression (\ref{(13)}) under the conditions
(\ref{(9)}) and (\ref{(10)}) we apply the Lagrange
method.

Consider the function
\begin{eqnarray*}
&&\Psi (C_0(z),C_1(z),...,C_N(z),p_0(z),...,p_{m-2}(z),d(z))\\
&&\qquad=\left\| \ell \right\|^2  -
2(-1)^m \sum\limits_{\alpha=0}^{m-2}p_{\alpha}(z)
\left({\ell,x^{\alpha}}\right)- 2(-1)^m d(z)
\left(\ell,\E^{-x}\right).
\end{eqnarray*}
Equating to 0 the partial derivatives of the function $\Psi$
by $C_\beta(z)$ $(\beta =\overline{0,N})$, $p_0,p_1,...,p_{m-2}$ and $d$, we get the following
system of linear equations
\begin{eqnarray}
&&\sum\limits_{\gamma=0}^N C_\gamma(z) G_m(x_\beta   - x_\gamma  )
+ \sum\limits_{\alpha=0}^{m-2}p_{\alpha}(z)x_{\beta}^{\alpha}+
d(z)\E^{-x_{\beta}} = G_m(z-x_{\beta}),\label{(14)}\\
&&\beta  =
0,1,...,N,\nonumber \\
&& \sum\limits_{\gamma=0}^N C_\gamma(z)
\ x_{\gamma}^{\alpha}=z^\alpha,\
\alpha=0,1,...,m-2,\label{(15)}\\
&&\sum\limits_{\gamma=0}^NC_\gamma(z)\ \E^{-x_{\gamma}}=\E^{-z},
\label{(16)}
\end{eqnarray}
where $G_m(x)$ is defined by equality (\ref{(12)}).

The system (\ref{(14)})-(\ref{(16)}) has a unique solution
and this solution gives the minimum to $\left\|\ell\right\|^2
$ under the conditions (\ref{(15)}) and (\ref{(16)}) when $N+1\geq m$.

The uniqueness of the solution of the system
(\ref{(14)})--(\ref{(16)}) is proved as the uniqueness of the solution of the system (26)--(28) of the work
\cite{ShadHay14}.

Therefore, in fixed values of the nodes $x_\beta$ the square of
the norm of the error functional $\ell$, being quadratic
function of the coefficients $C_\beta(z)$, has a unique minimum in
some concrete value $C_\beta(z)= \mathring{C}_{\beta }(z)$.

As it was said in the first section the interpolation formulas with the
coefficients $ \mathring{C}_{\beta }(z)$, corresponding
to this minimum in fixed nodes $x_\beta$, is called \emph{the
optimal interpolation formula} and $\mathring{C}_{\beta}(z)$ are called \emph{the optimal coefficients}.

\begin{remark}\label{Rem1}
It should be noted that by integrating both sides of the system (\ref{(14)})-(\ref{(16)}) by $z$ from 0 to 1
we get the system (26)-(28) of the work \cite{ShadHay14}. This means that by integrating the optimal interpolation formula (\ref{(1)}) in  the space $W_2^{(m,m-1)}(0,1)$ we get the optimal quadrature formula of the form (1) in the same space (see \cite{ShadHay14}).
\end{remark}

\begin{remark}\label{Rem2}
It is clear from the system (\ref{(14)})-(\ref{(16)}) that for the optimal coefficients the following are true
\begin{eqnarray*}
\mathring{C}_{\beta}(x_\gamma)&=&\delta_{\beta\gamma},\  \gamma=0,1,...,N,\ \ \beta=0,1,...,N,\\
p_{\alpha}(x_\gamma)&=&0,\ \ \gamma=0,1,...,N,
\end{eqnarray*}
where $\delta_{\beta\gamma}$ is the Kronecker symbol.
\end{remark}

Below for convenience the optimal coefficients $ \mathring{C}_{\beta}(z)$ we remain as $C_\beta(z)$.

\section{The algorithm for computation of coefficients of optimal interpolation formulas}

In the present section we give the algorithm for solution of the
system (\ref{(14)})-(\ref{(16)}). Here we use similar
method suggested by S.L. Sobolev \cite{SobVas,Sob77} for finding
the coefficients of optimal quadrature formulas in the space
$L_2^{(m)}(0,1)$. Below mainly is used the concept of discrete
argument functions and operations on them. The theory of discrete
argument functions is given in \cite{Sob74,SobVas}. For
completeness we give some definitions about functions of discrete
argument.

Assume that the nodes $x_\beta$ are equal spaced, i.e. $x_\beta=
h\beta,$ $h = {1 \over N}$, $N = 1,2,...$.

\begin{definition}
  The function $\varphi (h\beta )$ is a
\emph{function of discrete argument} if it is given on some set of
integer values of $\beta$.
\end{definition}

\begin{definition}
\emph{The inner product} of two discrete argument
functions $\varphi(h\beta )$ and $\psi (h\beta )$ is given by
$$
\left[ {\varphi(h\beta),\psi(h\beta) } \right] =
\sum\limits_{\beta  =  - \infty }^\infty  {\varphi (h\beta ) \cdot
\psi (h\beta )},
$$
if the series on the right hand side converges absolutely.
\end{definition}

\begin{definition}
\textit{The convolution} of two functions
$\varphi(h\beta )$ and $\psi (h\beta )$ is the inner product
$$
\varphi (h\beta )*\psi (h\beta ) = \left[ {\varphi (h\gamma ),\psi
(h\beta  - h\gamma )} \right] = \sum\limits_{\gamma  =  - \infty
}^\infty  {\varphi (h\gamma ) \cdot \psi (h\beta  - h\gamma )}.
$$
\end{definition}

Now we turn on to our problem.

Suppose that $C_\beta(z)=0$  when $\beta  < 0$ and $\beta  > N$.
Using above mentioned definitions we rewrite the system
(\ref{(14)})-(\ref{(16)})  in the convolution
form
\begin{eqnarray}
&&G_m(h\beta )*C_\beta(z)+\sum\limits_{\alpha=0}^{m-2}p_{\alpha}(z)(h\beta)^{\alpha} + d(z)\E^{-h\beta} =
G_m(z-h\beta),\label{(17)}\\
&& \beta  = 0,1,...,N, \nonumber \\
&& C_{\beta}(z)=0,\quad\quad \beta<0,\ \  \beta>N\label{(18)}\\
&&\sum\limits_{\beta=0}^N C_\beta(z)\cdot
(h\beta)^{\alpha}=z^{\alpha},\ \ \  \alpha=0,1,...,m-2,
\label{(19)}\\
&&\sum\limits_{\beta=0}^N C_\beta(z)\cdot \E^{-h\beta} =
\E^{-z}.\label{(20)}
\end{eqnarray}

Thus we have the following problem.

\begin{problem}
Find the discrete functions $C_\beta(z)$, $\beta=0,1,...,N$, $p_{\alpha}(z)$, $\alpha=0,1,...,m-2$ and $d(z)$ which
satisfy the system (\ref{(17)})-(\ref{(20)}).
\end{problem}

Further we investigate Problem 3 which is equivalent to Problem 2.
Instead of $C_\beta(z)$ we introduce the following functions
\begin{eqnarray}
v_m(h\beta ) &=& G_m(h\beta )*C_\beta(z),\label{(21)}\\
u_m(h\beta)&=&v_m\left({h\beta}\right)+\sum\limits_{\alpha=0}^{m-2}p_{\alpha}(z)(h\beta)^{\alpha}+ d(z)
\E^{-h\beta}.\label{(22)}
\end{eqnarray}
Now we should express the coefficients
$C_\beta(z)$ by the function $u(h\beta)$. For this we use the operator $D_m(h\beta)$ which satisfies the equality
\begin{equation}
D_m(h\beta )*G_m(h\beta ) =\delta_{d}(h\beta) ,\label{(23)}
\end{equation}
where $\delta_d(h\beta)$ is equal to 0 when $\beta \ne 0$ and is
equal to 1 when $\beta  = 0$, i.e. $\delta_d(h\beta)$ is the
discrete delta-function.

In \cite{ShadHay04a,ShadHay04b} the
operator $D_m(h\beta )$  which satisfies
equation (\ref{(23)}) is constructed and its some properties
are studied.

The following theorems are proved in works
\cite{ShadHay04a,ShadHay04b}.

\begin{theorem}\label{THM4.1}
The discrete analogue of the differential operator ${{\d^{2m}}
\over {\d x^{2m} }}-{{\d^{2m-2}} \over {\d x^{2m-2} }}$ satisfying the
equation (\ref{(23)}) has the form
\begin{equation}
D_m(h\beta)=\frac{1}{p_{2m-2}}\left\{
\begin{array}{ll}
\sum\limits_{k=1}^{m-1}A_k \lambda_k^{|\beta|-1},& |\beta|\geq 2,\\
-2\E^h+\sum\limits_{k=1}^{m-1}A_k,& |\beta|=1,\\
2C+\sum\limits_{k=1}^{m-1}\frac{A_k}{\lambda_k},& \beta=0,
\end{array}
\right.\label{(24)}
\end{equation}
where
\begin{eqnarray}
C&=&1+(2m-2)\E^h+\E^{2h}+\frac{\E^h\cdot
p_{2m-3}}{p_{2m-2}},\label{(25)}\\
A_k&=&\frac{2(1-\lambda_k)^{2m-2}[\lambda_k(\E^{2h}+1)-\E^h(\lambda_k^2+1)]p_{2m-2}}
{\lambda_k\ \mathcal{P}_{2m-2}'(\lambda_k)},\label{(26)}\\
\mathcal{P}_{2m-2}(\lambda)&=&\sum\limits_{s=0}^{2m-2}p_s\lambda^s\nonumber
\\
&=&(1-\E^{2h})(1-\lambda)^{2m-2}-
2(\lambda(\E^{2h}+1)-\E^h(\lambda^2+1)) \label{(27)}\\
&& \times
\left[h(1-\lambda)^{2m-4}+\frac{h^3(1-\lambda)^{2m-6}}{3!}E_2(\lambda)+...+
\frac{h^{2m-3}E_{2m-4}(\lambda)}{(2m-3)!}\right],\nonumber
\end{eqnarray}
$p_{2m-2},\ p_{2m-3}$ are the coefficients of
the polynomial $\mathcal{P}_{2m-2}(\lambda)$ defined  by equality
(\ref{(27)}), $\lambda_k$ are the roots of the polynomial
$\mathcal{P}_{2m-2}(\lambda)$ which absolute values less than 1,
$E_k(\lambda)$ is the Euler-Frobenius polynomial of degree $k$
(the definition of the Euler-Frobenius polynomial is given, for
example, in \cite{SobVas}).
\end{theorem}

\begin{theorem}
\label{THM4.2} The discrete analogue $D_m(h\beta)$ of the
differential operator ${{\d^{2m}} \over {\d x^{2m} }}-{{\d^{2m-2}}
\over {\d x^{2m-2} }}$ satisfies the following equalities

1) $D_m(h\beta)*\E^{h\beta}=0,$

2) $D_m(h\beta)*\E^{-h\beta}=0,$

3) $D_m(h\beta)*(h\beta)^n=0,$ for $n=0,1,...,2m-3$,

4) $D_m(h\beta)*G_m(h\beta)=\delta_d(h\beta),$\\
here $G_m(h\beta)$ is the function of discrete argument
corresponding to the function $G_m(x)$ defined by equality
(\ref{(12)}) and $\delta_d(h\beta)$ is the discrete delta-function.
\end{theorem}

Then taking into account (\ref{(22)}), (\ref{(23)}), using
Theorems \ref{THM4.1} and \ref{THM4.2}, for optimal coefficients we
have
\begin{equation}
C_\beta(z)=D_m(h\beta )*u_m(h\beta ).\label{(28)}
\end{equation}

Thus if we find the function $u_m(h\beta )$ then the optimal
coefficients $C_{\beta}(z)$ will be found from equality
(\ref{(28)}).

In order to calculate the convolution (\ref{(28)}) it is required to
find the representation of the function $u_m(h\beta )$ for all
integer values of $\beta$. From equality (\ref{(17)}) we get that
$u_m(h\beta ) = G_m(z-h\beta )$ when $h\beta \in [0,1]$. Now we find the representation of the function $u_m(h\beta )$ when
$\beta < 0$ and $\beta>N$.

Since $C_\beta(z)= 0$ when $h\beta \notin [0,1]$ then
$$
C_\beta(z)   = D_m(h\beta )*u_m(h\beta ) =
0,\,\,\,\, h\beta \notin [0,1].
$$

Now we calculate the convolution $v_m(h\beta ) =
G_m(h\beta)*C_\beta(z)$ when $h\beta \notin [0,1]$.

Suppose $\beta  \leq 0$ then taking into account equalities
(\ref{(12)}), (\ref{(19)}) and (\ref{(20)}), we have
\begin{eqnarray*}
 v_m(h\beta)&=& \sum\limits_{\gamma  =  - \infty }^\infty
{C_\gamma(z)
\,G_m(h\beta  - h\gamma )  }\\
&=&- {1 \over 2}\sum\limits_{\gamma  = 0}^N {C_\gamma(z)  } \left(
{{{\E^{h\beta  - h\gamma }  - \E^{ - h\beta  + h\gamma } } \over 2}
- \sum\limits_{k = 1}^{m - 1} {{{\left( {h\beta  - h\gamma }
\right)^{2k - 1} } \over {(2k - 1)!}}} } \right) \\
&=&-{{\E^{h\beta } } \over 4}\ \E^{-z}+D\E^{-h\beta }+Q_{2m -
3}(h\beta ) + R_{m - 2}(h\beta ).
\end{eqnarray*}
Thus when $\beta\leq 0$ we get
\begin{equation}
v_m(h\beta)=-{{\E^{h\beta } } \over 4}\ \E^{-z}+D\E^{-h\beta }+Q_{2m -
3}(h\beta) + R_{m - 2}(h\beta), \label{(29)}
\end{equation}
where
\begin{equation}
\begin{array}{rl}
Q_{2m - 3}(h\beta )=&\sum\limits_{i=0}^{2m-3}q_i(h\beta)^i\\[5mm]
=&{1 \over
2} \left({\sum\limits_{k = 1}^{\left[ {{{m + 1} \over 2}} \right]
- 1} {\sum\limits_{\alpha = 0}^{2k - 1} {{{(h\beta )^{2k - 1 -
\alpha } (-1)^\alpha z^\alpha } \over {(2k - 1 - \alpha )! \cdot
\alpha!}}    } } } \right. +\left. {\sum\limits_{k = \left[ {{{m +
1} \over 2}} \right]}^{m - 1}
 {\sum\limits_{\alpha  = 0}^{m - 2} {{{(h\beta )^{2k - 1 - \alpha }
 ( - 1)^\alpha z^\alpha} \over {(2k - 1 - \alpha )! \cdot \alpha!}}} } } \right)
 \end{array}
\label{(30)}
\end{equation}
is the polynomial of degree $2m-3$ with respect to $(h\beta)$,
\begin{equation}
\begin{array}{rl}
R_{m - 2}(h\beta )=&\sum\limits_{i=0}^{m-2}r_i(h\beta)^i\\
=&{1 \over 2}\sum\limits_{k= \left[{{{m + 1}
\over 2}} \right]}^{m - 1} {\sum\limits_{\alpha  = m - 1}^{2k - 1}
{{{(h\beta )^{2k - 1 - \alpha } (-1)^\alpha  } \over {(2k - 1 -
\alpha )! \cdot \alpha !}}} } \sum\limits_{\gamma  = 0}^N
{C_\gamma(z)} (h\gamma )^\alpha
\end{array}
\label{(31)}
\end{equation}
is unknown polynomial of degree $m-2$ of $(h\beta)$,
\begin{equation}
D={1 \over 4}\sum\limits_{\gamma  = 0}^N {C_\gamma(z)  \E^{h\gamma
} }.\label{(32)}
\end{equation}

Similarly, in the case $\beta\geq N$ for the convolution
$v_m(h\beta) = G_m(h\beta)*C_\beta(z)$ we obtain
\begin{equation}
v_m(h\beta )= {{\E^{h\beta } } \over 4}\ \E^{-z}-D\E^{ - h\beta }-
Q_{2m-3}(h\beta)-R_{m-2}(h\beta ).\label{eq.(4.16)}
\end{equation}

We denote
\begin{eqnarray}
R_{m - 2}^{-}(h\beta )= P_{m - 2}(h\beta ) + R_{m - 2} (h\beta
),\,\,\,\,\,\,a^ -   = d + D,\label{eq.(4.17)}\\
R_{m - 2}^{+}(h\beta ) = P_{m - 2} (h\beta ) - R_{m - 2} (h\beta
),\,\,\,\,\,\,a^ +   = d - D,\label{eq.(4.18)}
\end{eqnarray}
where $R_{m-2}^-(h\beta)=\sum\limits_{i=0}^{m-2}r_i^-(h\beta)^i$ and $R_{m-2}^+(h\beta)=\sum\limits_{i=0}^{m-2}r_i^+(h\beta)^i$.\\
Then taking into account (\ref{(22)}), (\ref{(29)}),
(\ref{eq.(4.16)}) and the last notations (\ref{eq.(4.17)}), (\ref{eq.(4.18)}) we get the following problem.

\begin{problem}
Find the solution of the equation
\begin{equation}
D_m(h\beta )*u_m(h\beta) = 0,\,\, h\beta  \notin
[0,1],\label{eq.(4.19)}
\end{equation}
which has the form
\begin{equation}
u_m(h\beta ) = \left\{
\begin{array}{ll}
-\frac{\E^{h\beta}}{4}\ \E^{-z}+a^-\E^{-h\beta}+
Q_{2m-3}(h\beta)+R_{m-2}^{-}(h\beta),& \beta\leq 0, \\
G_m(z-h\beta ),& 0 \le \beta  \le N, \\
\frac{\E^{h\beta}}{4}\ \E^{-z}+a^+\E^{-h\beta}-
Q_{2m-3}(h\beta)+R_{m-2}^{+}(h\beta),& \beta\geq N.\\
\end{array}
\right.\label{eq.(4.20)}
\end{equation}
Here $R_{m-2}^{-}(h\beta)$ and $R_{m-2}^{+}(h\beta)$ are unknown
polynomials of degree $m-2$ with respect to $h\beta$, $a^-$ and
$a^+$ are unknown constants.
\end{problem}

If we find $R_{m-2}^{-}(h\beta)$, $R_{m-2}^{+}(h\beta)$, $a^-$ and
$a^+$ then from (\ref{eq.(4.17)}), (\ref{eq.(4.18)}) we have
$$
P_{m-2}(h\beta)=\frac{1}{2}\left(R_{m-2}^{-}(h\beta)+R_{m-2}^{+}(h\beta)\right),\
\ d=\frac{1}{2}(a^-+a^+),
$$
$$
R_{m-2}(h\beta)=\frac{1}{2}\left(R_{m-2}^{-}(h\beta)-R_{m-2}^{+}(h\beta)\right),\
\ D=\frac{1}{2}(a^--a^+),
$$

Unknowns $R_{m-2}^{-}(h\beta)$, $R_{m-2}^{+}(h\beta)$, $a^-$ and
$a^+$  we will find from the equation (\ref{eq.(4.19)}), using the
function $D_m(h\beta)$ defined by (\ref{(24)}). Then we obtain
explicit form of the function $u_m(h\beta )$ and find the optimal
coefficients $C_\beta(z)$ ($\beta=0,1,...N$) from (\ref{(28)}).

Thus Problem 4 and respectively Problems 3 and 2 will be solved.

In the next section we realize this algorithm for computation of
coefficients $C_{\beta}(z)$ of optimal interpolation formulas
(\ref{(1)}).

\section{Computation of coefficients of optimal interpolation formulas of the form (\ref{(1)})}

In this section we give the solution of Problem 4 for any $m\in \mathbb{N}$ and $N\geq m-1$.
We consider the cases $m=1$ and $m\geq 2$ separately.

\medskip

First we consider \textbf{the case $m=1$}. For this case we have the following
results

\begin{theorem}\label{THM5.1}
Coefficients of the optimal interpolation formula (\ref{(1)})
with equal spaced nodes in the space $W_2^{(1,0)}(0,1)$ have the
following form
\begin{equation}\label{eq.(5.1)}
\begin{array}{ll}
\mathring{C}_{\beta}(z)=&
\frac{1}{2(1-\E^{2h})}\Bigg[\mathrm{sgn}(z-h\beta-h)\cdot
(\E^{h\beta+2h-z}-\E^{z-h\beta})\\
&\qquad+\mathrm{sgn}(z-h\beta+h)\cdot
(\E^{h\beta-z}-\E^{z-h\beta+2h})\\
&\qquad+(1+\E^{2h})\cdot
\mathrm{sgn}(z-h\beta)\cdot
(\E^{z-h\beta}-\E^{h\beta-z})\Bigg],\ \ \beta=0,1,...,N.
\end{array}
\end{equation}

\end{theorem}

\begin{proof}
Assume $m=1$. In this case the function $u_m(h\beta)$, given by
equality (\ref{eq.(4.20)}), takes the form
\begin{equation}
u_1(h\beta)=\left\{
\begin{array}{ll}
-\frac{1}{4}\E^{h\beta-z}+a^-\E^{-h\beta},& \beta\leq 0,\\[2mm]
\frac{1}{4}\mathrm{sgn}(z-h\beta)\cdot (\E^{z-h\beta}-\E^{h\beta-z}),& 0\leq \beta\leq N,\\[2mm]
\frac{1}{4}\E^{h\beta-z}+a^+\E^{-h\beta},& \beta\geq N\\
\end{array}
\right. \label{eq.(5.2)}
\end{equation}
and satisfies the equation
\begin{equation}
D_1(h\beta)*u_1(h\beta)=0\mbox{ for } \beta<0\mbox{ and }  \beta>N,
\label{eq.(5.3)}
\end{equation}
where $D_1(h\beta)$ is defined by (\ref{(24)}) when $m=1$:
\begin{equation}
D_1(h\beta)=\frac{2}{1-\E^{2h}}\left\{
\begin{array}{ll}
0, &|\beta|\geq 2,\\
-\E^h,& |\beta|=1,\\
1+\E^{2h}, & \beta=0.
\end{array}
\right. \label{eq.(5.4)}
\end{equation}
In (\ref{eq.(5.2)})   $a^-$ and $a^+$ are unknowns. From
equation (\ref{eq.(5.3)}), using (\ref{eq.(5.2)}) and
(\ref{eq.(5.4)}), for $\beta=-1$ and $\beta=N+1$ we respectively
get
\begin{equation} \label{eq.(5.5)}
 a^-=\E^z/4,\ \
a^+=-\E^z/4.
\end{equation}
Then from (\ref{eq.(4.17)}), (\ref{eq.(4.18)}) taking into account
(\ref{eq.(5.5)}) we obtain
\begin{equation} \label{eq.(5.6)}
d=0,\ \ D=\E^z/4.
\end{equation}
Thus, putting (\ref{eq.(5.5)}) to (\ref{eq.(5.2)}) we have the
following explicit form of the function $u_1(h\beta)$:
\begin{equation}
u_1(h\beta)=\left\{
\begin{array}{ll}
-\frac{1}{4}(\E^{h\beta-z}-\E^{z-h\beta}),& \beta\leq 0,\\[2mm]
\frac{1}{4}\mathrm{sgn}(z-h\beta)\cdot (\E^{z-h\beta}-\E^{h\beta-z}),& 0\leq \beta\leq N,\\[2mm]
\frac{1}{4}(\E^{h\beta-z}-\E^{z-h\beta}),& \beta\geq N.\\
\end{array}
\right. \label{eq.(5.7)}
\end{equation}
Now from equality (\ref{(28)}) we get
$$
C_{\beta}(z)=D_1(h\beta)*u_1(h\beta),\ \ \beta=0,1,...,N.
$$
Hence, using (\ref{eq.(5.4)}), (\ref{eq.(5.7)}) and computing the
convolution $D_1(h\beta)*u_1(h\beta)$ for $\beta=0,1,...,N$ we get
(\ref{eq.(5.1)}). Theorem \ref{THM5.1} is proved.
\end{proof}

Now we calculate the norm of the error functional $\ell$ on the
space $W_2^{(1,0)}(0,1)$.

For $m=1$ taking into account (\ref{eq.(5.6)}) from
(\ref{(17)}) we have
$$
\sum\limits_{\gamma=0}^NC_{\gamma}(z)G_1(h\beta-h\gamma)=G_1(z-h\beta),\
\ \beta=0,1,...,N.
$$
Then from (\ref{(13)}) using the last equality we get
\begin{eqnarray*}
\left\| \mathring{\ell}\right\|^2&=&-\left(\sum\limits_{\beta=0}^NC_\beta(z)\left[\sum\limits_{\gamma=0}^NC_{\gamma}(z)
G_1(h\beta-h\gamma)-G_1(z-h\beta)\right]\right)+\sum\limits_{\beta=0}^NC_{\beta}(z)G_1(z-h\beta)\\
&=&\sum\limits_{\beta=0}^NC_{\beta}(z)G_1(z-h\beta).
\end{eqnarray*}
Hence taking into account (\ref{(12)}) we get the following

\begin{theorem}\label{THM5.2}
The square of the norm of the error functional
$\ell$ of the optimal interpolation formula (\ref{(1)}) on the
space $W_2^{(1,0)}(0,1)$ has the form
\begin{equation}\label{eq.(5.8)}
\left\| \mathring{\ell}
|W_2^{(1,0)*}(0,1)\right\|^2=\frac14\sum\limits_{\beta=0}^N\mathring{C}_{\beta}(z)\cdot
\mathrm{sgn}(z-h\beta)\cdot (\E^{z-h\beta}-\E^{h\beta-z}),
\end{equation}
where $\mathring{C}_{\beta}(z)$ are defined by (\ref{eq.(5.1)}).

\end{theorem}

\begin{remark}\label{Rem3}
We note that in \cite{Hay15} coincidence of the optimal interpolation formula of the form (\ref{(1)}) in the space $W_2^{(1,0)}(0,1)$  with the interpolation spline $S_1(x)$, constructed in Theorem 3.2 of the work \cite{ShadHay13}, minimizing the semi-norm in this space is shown. The spline $S_1(x)$, constructed in Theorem 3.2 of \cite{ShadHay13}, is used in \cite{Avez10} to
determine the total incident solar radiation at each time step and the temperature
dependence of the thermophysical properties of the air and water.
\end{remark}

Now we consider \textbf{the case $m\geq 2$}. In this case the following is true.

\begin{theorem}\label{THM5}
Coefficients of optimal interpolation formulas (\ref{(1)})
with equal spaced nodes $x_{\beta}=h\beta$ in the space $W_2^{(m,m-1)}(0,1)$ when $m\geq 2$ have the
following form
\begin{eqnarray*}
\mathring{C}_0(z)&=&\frac{1}{p}\Bigg[2CG_m(z)-2\E^h\bigg[G_m(z-h)-\frac{1}{4}\E^{-h-z}+a^-\E^h+\sum\limits_{i=0}^{2m-3}q_i(-h)^i+
\sum\limits_{i=0}^{m-2}r_i^-(-h)^i\bigg]\\
&&+\sum\limits_{k=1}^{m-1}\frac{A_k}{\lambda_k}\Bigg\{\sum\limits_{\gamma=0}^N\lambda_k^{\gamma}G_m(z-h\gamma)+M_k+\lambda_k^NN_k\Bigg\}\Bigg],\\
\mathring{C}_{\beta}(z)&=&\frac{1}{p}\Bigg[2CG_m(z-h\beta)-2\E^h\bigg[G_m(z-h(\beta-1))+G_m(z-h(\beta+1))\bigg]\\
&&
+\sum\limits_{k=1}^{m-1}\frac{A_k}{\lambda_k}\Bigg\{\sum\limits_{\gamma=0}^N\lambda_k^{|\beta-\gamma|}G_m(z-h\gamma)+
\lambda_k^{\beta}M_k+\lambda_k^{N-\beta}N_k\Bigg\}\Bigg],\\
\mathring{C}_N(z)&=&\frac{1}{p}\Bigg[2CG_m(z-1)-2\E^h\bigg[G_m(z-1+h)+\frac{1}{4}\E^{1+h-z}+a^+\E^{-1-h}\\
&&-\sum\limits_{i=0}^{2m-3}q_i(1+h)^i+
\sum\limits_{i=0}^{m-2}r_i^+(1+h)^i\bigg]\\
&&
+\sum\limits_{k=1}^{m-1}\frac{A_k}{\lambda_k}\Bigg\{\sum\limits_{\gamma=0}^N\lambda_k^{N-\gamma}G_m(z-h\gamma)+\lambda_k^NM_k+N_k\Bigg\}\Bigg],
\end{eqnarray*}
where
\begin{eqnarray*}
M_k&=&\frac{\lambda_k\E^{-z}}{4(\lambda_k-\E^h)}+\frac{a^-\lambda_k\E^h}{1-\lambda_k \E^h}+\sum\limits_{i=1}^{2m-3}q_i(-h)^i\sum\limits_{\nu=1}^i
\frac{\lambda_k^{\nu}\Delta^\nu0^i}{(1-\lambda_k)^{\nu+1}}+\frac{q_0\lambda_k}{1-\lambda_k}\\
&&\qquad\qquad\qquad +\sum\limits_{i=1}^{m-2}r_i^-(-h)^i\sum\limits_{\nu=1}^i\frac{\lambda_k^{\nu}\Delta^\nu 0^i}{(1-\lambda_k)^{\nu+1}}+
\frac{r_0^-\lambda_k}{(1-\lambda_k)},\\
N_k&=&\frac{\lambda_k\E^{1-z+h}}{4(1-\lambda_k\E^h)}+\frac{a^+\lambda_k}{\E(\E^h-\lambda_k)}-\sum\limits_{i=1}^{2m-3}q_i\sum\limits_{j=1}^i{i \choose j}
h^j\sum\limits_{\nu=1}^j\frac{\lambda_k^\nu\Delta^\nu 0^j}{(1-\lambda_k)^{\nu+1}}-\sum\limits_{i=0}^{2m-3}q_i\frac{\lambda_k}{1-\lambda_k}\\
&&\qquad\qquad\qquad +\sum\limits_{i=1}^{m-2}r_i^+\sum\limits_{j=1}^i{i\choose j}h^j\sum\limits_{\nu=1}^j\frac{\lambda_k^{\nu}\Delta^{\nu}0^j}{(1-\lambda_k)^{\nu+1}}
+\sum\limits_{i=0}^{m-2}\frac{r_i^+\lambda_k}{(1-\lambda_k)},\\
p&=&p_{2m-2}=1-\E^{2h}+2\E^h\left(h+\frac{h^3}{3!}+...+\frac{h^{2m-3}}{(2m-3)!}\right), \\
a^-&=&G_m(z)+\frac{1}{4}\E^{-z}-q_0-r_0^-,\\
a^+&=&\E\left[G_m(z-1)-\frac{1}{4}\E^{1-z}+\sum\limits_{i=0}^{2m-3}q_i-\sum\limits_{i=0}^{m-2}r_i^+\right],\\
\end{eqnarray*}
and $r_i^-$, $r_i^+$, $i=0,1,...,m-2$ satisfy the system (\ref{(491)}) and (\ref{(501)}) of $2m-2$ linear equations,
$\lambda_k,\ A_k$ and $C$ are given in Theorem \ref{THM4.1}, $q_i$ are defined by (\ref{(30)}).
\end{theorem}

\begin{proof}
Here we find explicit form of $u_m(h\beta)$ then from (\ref{(28)}), using $D_m(h\beta)$, we find optimal coefficients $C_{\beta}(z)$ .

\newpage

In order to find $u_m(h\beta)$ we should find unknowns $r_i^-,\ r_i^+,$ $i=0,1,...,m-2$ and $a^-$, $a^+$. First we find $a^-$ and $a^+$.
From (\ref{eq.(4.20)}) when $\beta=0$ and $\beta=N$ we get
\begin{eqnarray}
a^-&=&G_m(z)+\frac{1}{4}\E^{-z}-q_0-r_0^-,\label{(49)}\\
a^+&=&\E\left(G_m(z-1)-\frac{1}{4}\E^{1-z}+\sum\limits_{i=0}^{2m-3}q_i-\sum\limits_{i=0}^{m-2}r_i^+\right). \label{(50)}
\end{eqnarray}

Further, we find $2m-2$ unknowns $r_i^-,\ r_i^+,$ $i=0,1,...,m-2$.
For $r_i^-,\ r_i^+,$ $i=0,1,...,m-2$  using (\ref{eq.(4.20)}), (\ref{(49)}) and (\ref{(50)}) from equation (\ref{eq.(4.19)})
we have the following system of $2m-2$ linear equations
\begin{equation}
\begin{array}{l}
\sum\limits_{\gamma=1}^{\infty}D_m(h\beta+h\gamma)\bigg[\frac{\E^{-z}}{2}\sinh(h\gamma)+\E^{h\gamma}G_m(z)+\sum\limits_{i=0}^{2m-3}q_i(-h\gamma)^i-\E^{h\gamma}q_0\\
\qquad\qquad+\sum\limits_{i=0}^{m-2}r_i^-(-h\gamma)^i-\E^{h\gamma}r_0^-\bigg] +\sum\limits_{\gamma=0}^ND_m(h\beta-h\gamma)G_m(z-h\gamma)\\
\qquad\qquad+\sum\limits_{\gamma=1}^{\infty}D_m(h(N+\gamma-\beta))\bigg[\frac{\E^{1-z}}{2}\sinh(h\gamma)+\E^{-h\gamma}G_m(z-1)\\
\qquad\qquad+\sum\limits_{i=0}^{2m-3}q_i(\E^{-h\gamma}-(1+h\gamma)^i)+\sum\limits_{i=0}^{m-2}r_i^+((1+h\gamma)^i-\E^{-h\gamma})\bigg]=0,
\end{array}\label{(51)}
\end{equation}
where $\beta=-1,-2,...,-(m-1)$ and $\beta=N+1,N+2,...,N+m-1$.

From the system (\ref{(51)}) in the cases $\beta=-1,-2,...,-(m-1)$ replacing $\beta$ by $-\beta$, using (\ref{(24)})  and after some simplifications we have the following system of $m-1$ linear equations
\begin{equation}\label{(491)}
\sum\limits_{i=0}^{m-2}B_{\beta i}^-r_i^-+\sum\limits_{i=0}^{m-2}B_{\beta i}^+r_i^+=T_{\beta},\   \beta=1,2,...,m-1,  \\
\end{equation}
where
\begin{equation*}
\begin{array}{ll}
B_{\beta 0}^-=&\sum\limits_{k=1}^{m-1}\frac{A_k}{\lambda_k}\sum\limits_{\gamma=1}^{\infty}\lambda_k^{|\beta-\gamma|}(1-\E^{h\gamma})
+2C(1-\E^{h\beta})-2\E^h(2-\E^{h(\beta-1)}-\E^{h(\beta+1)}),\\
B_{\beta i}^-=&(-h)^i\left[2C\beta^i-2\E^h((\beta-1)^i+(\beta+1)^i)+\sum\limits_{k=1}^{m-1}\frac{A_k}{\lambda_k}\sum\limits_{\gamma=1}^{\infty}
\lambda_k^{|\beta-\gamma|}\gamma^i\right],\\
B_{\beta 0}^+=&\sum\limits_{k=1}^{m-1}\frac{A_k\lambda_k^{N+\beta}(\E^h-1)}{(1-\lambda_k)(\E^h-\lambda_k)},\\
B_{\beta i}^+=&\sum\limits_{k=1}^{m-1}\frac{A_k}{\lambda_k}\lambda_k^{N+\beta}\left(\sum\limits_{j=1}^i{i\choose j}h^j\sum\limits_{\nu=1}^j\frac{\lambda_k^\nu\Delta^\nu 0^j}{(1-\lambda_k)^{\nu+1}}+\frac{\lambda_k(\E^h-1)}{(1-\lambda_k)(\E^h-\lambda_k)}\right),\\
T_\beta=&-\Bigg\{\E^{-z}\bigg[C\sinh(h\beta)-\E^h\left[\sinh(h(\beta-1))+\sinh(h(\beta+1))\right]\bigg]+G_m(z)
\left[2C\E^{h\beta}-2\E^h(\E^{h(\beta-1)}+\E^{h(\beta+1)})\right]\\
&+\sum\limits_{i=1}^{2m-3}q_i(-h)^i\bigg[2C\beta^i-2\E^h((\beta-1)^i+(\beta+1)^i)\bigg]
+q_0\bigg[2C(1-\E^{h\beta})-2\E^h(2-\E^{h(\beta-1)}-\E^{h(\beta+1)})\bigg]\Bigg\}\\
&-\sum\limits_{k=1}^{m-1}\frac{A_k}{\lambda_k}\Bigg\{\sum\limits_{\gamma=0}^N\lambda_k^{\gamma+\beta}G_m(z-h\gamma)
+\sum\limits_{\gamma=1}^{\infty}\lambda_k^{|\beta-\gamma|}\bigg[\frac{\E^{-z}}{2}\sinh(h\gamma)+\E^{h\gamma}G_m(z)+
\sum\limits_{i=1}^{2m-3}q_i(-h\gamma)^i+q_0(1-\E^{h\gamma})\bigg]\\
&+\lambda_k^{N+\beta}\bigg[\frac{\E^{1-z}\lambda_k\sinh(h)}{2(\lambda_k^2+1-2\lambda_k\cosh(h))}+\frac{\lambda_k}{\E^h-\lambda_k}G_m(z-1)+
\sum\limits_{i=0}^{2m-3}q_i\left(\frac{\lambda_k(1-\E^h)}{(\E^h-\lambda_k)(1-\lambda_k)}-\sum\limits_{j=1}^i{i\choose j}h^j\sum\limits_{\nu=1}^j\frac{\lambda_k^\nu\Delta^\nu 0^j}{(1-\lambda_k)^{\nu+1}}\right)\bigg]\Bigg\},\\
&\beta=1,2,...,m-1, \ \ i=1,2,...,m-2
\end{array}
\end{equation*}

Now, from (\ref{(51)}) in the cases $\beta=N+1,N+2,...,N+m-1$ replacing $\beta$ by $N+\beta$, using (\ref{(24)}), doing some calculations
we get the next system of $m-1$ linear equations
\begin{equation}\label{(501)}
\sum\limits_{i=0}^{m-2}A_{\beta i}^-r_i^-+\sum\limits_{i=0}^{m-2}A_{\beta i}^+r_i^+=S_{\beta},\  \beta=1,2,...,m-1,
\end{equation}
where
\begin{equation*}
\begin{array}{ll}
A_{\beta 0}^-=&\sum\limits_{k=1}^{m-1}\frac{A_k\lambda_k^{N+\beta}(1-\E^h)}{(1-\lambda_k)((1-\E^h\lambda_k)},\\
A_{\beta i}^-=&(-h)^i\sum\limits_{k=1}^{m-1}\frac{A_k}{\lambda_k}\lambda_k^{N+\beta}\sum\limits_{\nu=1}^i\frac{\lambda_k^\nu\Delta^\nu 0^i}{(1-\lambda_k)^{\nu+1}},\\
A_{\beta 0}^+=&\sum\limits_{k=1}^{m-1}\frac{A_k}{\lambda_k}\sum\limits_{\gamma=1}^{\infty}\lambda_k^{|\beta-\gamma|}(1-\E^{-h\gamma})+2C(1-\E^{-h\beta})
  -2\E^h\bigg[2-\E^{-h(\beta-1)}-\E^{-h(\beta+1)}\bigg],\\
A_{\beta i}^+=&\sum\limits_{k=1}^{m-1}\frac{A_k}{\lambda_k}\sum\limits_{\gamma=1}^{\infty}\lambda_k^{|\beta-\gamma|}\bigg[\sum\limits_{j=1}^i{i\choose j}(h\gamma)^j+1-\E^{-h\gamma}\bigg]+2C\bigg[\sum\limits_{j=1}^i{i \choose j}(h\beta)^j+1-\E^{-h\beta}\bigg]\\
&-2\E^h\bigg[\sum\limits_{j=1}^i{i\choose j}\left((h(\beta-1))^j+(h(\beta+1))^j\right)+2-\E^{-h(\beta-1)}-\E^{-h(\beta+1))}\bigg],\\
S_\beta=&-\Bigg\{\E^{1-z}\bigg[C\sinh(h\beta)-\E^h\left[\sinh(h(\beta-1))+\sinh(h(\beta+1))\right]\bigg]+G_m(z-1)
\left[2C\E^{-h\beta}-2\E^h(\E^{-h(\beta-1)}+\E^{-h(\beta+1)})\right]\\
&+\sum\limits_{i=0}^{2m-3}q_i\bigg[2C\bigg(\E^{-h\beta}-1-\sum\limits_{j=1}^i{i\choose j}(h\beta)^j\bigg)
-2\E^{h}\bigg(\E^{-h(\beta-1)}+\E^{-h(\beta+1)}-2-\sum\limits_{j=1}^i{i\choose j}\left((h(\beta-1))^j+(h(\beta+1))^j\right)\bigg)\bigg]\Bigg\}\\
&-\sum\limits_{k=1}^{m-1}\frac{A_k}{\lambda_k}\Bigg\{\sum\limits_{\gamma=0}^N\lambda_k^{N+\beta-\gamma}G_m(z-h\gamma)
+\lambda_k^{N+\beta}\bigg[\frac{\E^{-h}\lambda_k \sinh(h)}{2(\lambda_k^2+1-2\lambda_k\cosh(h))}+G_m(z)\frac{\lambda_k\E^h}{1-\E^h\lambda_k}+\sum\limits_{i=1}^{2m-3}q_i(-h)^i\sum\limits_{\nu=1}^i
\frac{\lambda_k^{\nu}\Delta^\nu 0^i}{(1-\lambda_k)^{\nu+1}}\\
&+\frac{q_0\lambda_k(1-\E^h)}{(1-\lambda_k)(1-\lambda_k\E^h)}\bigg]+\sum\limits_{\gamma=1}^{\infty}\lambda_k^{|\beta-\gamma|}
\bigg[\frac{\E^{1-z}}{2}\sinh(h\gamma)+\E^{-h\gamma}G_m(z-1)+\sum\limits_{i=0}^{2m-3}q_i\left(\E^{-h\gamma}-1-\sum\limits_{j=1}^i{i\choose j}(h\gamma)^j\right)\bigg].\\
&\beta=1,2,...,m-1, \ \ i=1,2,...,m-2,
\end{array}
\end{equation*}

Further from (\ref{(28)}), using (\ref{(24)}) and (\ref{eq.(4.20)}) we get the optimal coefficients $C_{\beta}$, $\beta=0,1,...,N$, which are given in the assertion of the Theorem.

Theorem \ref{THM5} is proved
\end{proof}

For $m=2$ from Theorem \ref{THM5} we get the following result which is Theorem 3 of \cite{Hay14}.

\begin{corollary}\label{Cor1}(Theorem 3 of \cite{Hay14}).
Coefficients of optimal interpolation formula (\ref{(1)})
with equal spaced nodes in the space $W_2^{(2,1)}(0,1)$ have the following form
\begin{eqnarray*}
\mathring{C}_0(z)&=&\frac{1}{p}\Bigg\{2C G_2(z)-2\E^h\bigg[G_2(z-h)-\frac{1}{4}\E^{-h-z}+a^-\E^h-\frac{1}{2}h+r_0^-\bigg]\\
&&
+\frac{A_1}{\lambda_1}\bigg[\sum\limits_{\gamma=0}^N\lambda_1^{\gamma}G_2(z-h\gamma)+M_1+\lambda_1^NN_1\bigg]\Bigg\},\\
\mathring{C}_\beta(z)&=&\frac{1}{p}\Bigg\{2CG_2(z-h\beta)-2\E^h\bigg[G_2(z-h(\beta-1))+G_2(z-h(\beta+1))\bigg]\\
&&+\frac{A_1}{\lambda_1}\bigg[\sum\limits_{\gamma=0}^N\lambda_1^{|\beta-\gamma|}G_2(z-h\gamma)+\lambda_1^{\beta}M_1+\lambda_1^{N-\beta}N_1\bigg]\Bigg\},\ \ \beta=1,2,...,N-1,\\
\mathring{C}_N(z)&=&\frac{1}{p}\Bigg\{2CG_2(z-1)-2\E^h\bigg[G_2(z-1+h)
+\frac{\E^{1+h}}{4\E^z}+\frac{a^+}{\E^{1+h}}-\frac{1}{2}(1+h)+r_0^+\bigg]\\
&&
+\frac{A_1}{\lambda_1}\bigg[\sum\limits_{\gamma=0}^N\lambda_1^{N-\gamma}G_2(z-h\gamma)+\lambda_1^NM_1+N_1\bigg]\Bigg\},
\end{eqnarray*}
where
\begin{eqnarray*}
M_1&=&\frac{\lambda_1\E^{-z}}{4(\lambda_1-\E^h)}+\frac{a^{-}\lambda_1\E^{h}}{1-\lambda_1\E^h}-
\frac{h\lambda_1}{2(1-\lambda_1)^2}+\frac{r_0^{-}\lambda_1}{1-\lambda_1},\\
N_1&=&\frac{\lambda_1\E^{1-z+h}}{4(1-\lambda_1\E^h)}+\frac{a^{+}\lambda_1}{\E(\E^h-\lambda_1)}-\frac{h\lambda_1}{2(1-\lambda_1)^2}-
\frac{\lambda_1}{2(1-\lambda_1)}+\frac{r_0^{+}\lambda_1}{1-\lambda_1},\\
\lambda_1&=&\frac{h(\E^{2h}+1)-\E^{2h}+1-(\E^h-1)\sqrt{h^2(\E^h+1)^2+2h(1-\E^{2h})}}{1-\E^{2h}+2h\E^h},
\end{eqnarray*}
$$
\begin{array}{ll}
G_2(z)=\displaystyle\frac{\mathrm{sgn}z}{2}\bigg(\sinh z-z\bigg),&\displaystyle p=1-\E^{2h}+2h\E^{h},\\
C=\displaystyle 1+2\E^h+\E^{2h}-\frac{\E^h(\lambda_1^2+1)}{\lambda_1}, & A_1=\displaystyle \frac{2(\lambda_1-1)(\lambda_1(\E^{2h}+1)-\E^h(\lambda_1^2+1))}{\lambda_1+1}\\
\displaystyle a^{-}=G_2(z)+\frac{\E^{-z}}{4}-r_0^{-},&\displaystyle a^{+}=\E\bigg(G_2(z-1)-\frac{\E^{1-z}}{4}-r_0^{+}+\frac{1}{2}\bigg),\\
\displaystyle r_0^{-}=\frac{T_1A_{10}^{+}-S_1B_{10}^{+}}{B_{10}^{-}A_{10}^{+}-B_{10}^{+}A_{10}^{-}},&
\displaystyle r_0^{+}=\frac{S_1B_{10}^{-}-T_1A_{10}^{-}}{B_{10}^{-}A_{10}^{+}-B_{10}^{+}A_{10}^{-}},\\
\end{array}
$$
here
\begin{eqnarray*}
B_{10}^{-}&=&2C(1-\E^h)-2\E^h(1-\E^{2h})+\sum\limits_{\gamma=1}^\infty A_1\lambda_1^{\gamma-2}(1-\E^{h\gamma}),\\
B_{10}^{+}&=&\frac{A_1\lambda_1^{N+1}(\E^h-1)}{(1-\lambda_1)(\E^h-\lambda_1)},\\
\end{eqnarray*}
\begin{eqnarray*}
A_{10}^{-}&=&A_1\lambda_1^N\sum\limits_{\gamma=1}^\infty \lambda_1^{\gamma}(1-\E^{h\gamma}),\\
A_{10}^{+}&=&2C(1-\E^{-h})-2\E^h(1-\E^{-2h})+\sum\limits_{\gamma=1}^\infty A_1\lambda_1^{\gamma-2}(1-\E^{-h\gamma}),\\
T_1&=&-\bigg[\E^{-z}(C\sinh(h)-\E^h\sinh(2h))+G_2(z)(2C\E^h-2\E^h(1+\E^{2h})-h(C-2\E^h))\bigg]\\
&&-\frac{A_1}{\lambda_1}\bigg[\sum\limits_{\gamma=0}^N\lambda_1^{\gamma+1}G_2(z-h\gamma)+\sum\limits_{\gamma=1}^\infty\lambda_1^{\gamma-1}
\bigg(\frac{\E^{-z}\sinh(h\gamma)}{2}+\E^{h\gamma}G_2(z)-\frac{h\gamma}{2}\bigg)\\
&&
+\lambda_1^{N+1}\bigg(\frac{\E^{1-z}\lambda_1\sinh(h)}{2(\lambda_1^2+1-2\lambda_1\cosh(h))}+
\frac{\lambda_1 G_2(z-1)}{\E^h-\lambda_1}+\frac{1}{2}\bigg(\frac{\lambda_1(1-\E^h)}{(\E^h-\lambda_1)(1-\lambda_1)}-
\frac{\lambda_1 h}{(1-\lambda_1)^2}\bigg)\bigg)\bigg],\\
S_1&=&-\bigg[\E^{1-z}(C\sinh(h)-\E^h\sinh(2h))+G_2(z-1)(2C\E^{-h}-2\E^h(1+\E^{-2h}))\\
&&
+C(\E^{-h}-1-h)-\E^h(\E^{-2h}-1-2h)\bigg]-\frac{A_1}{\lambda_1}\bigg[\sum\limits_{\gamma=0}^N\lambda_1^{N-\gamma+1}G_2(z-h\gamma)\\
&&
+\sum\limits_{\gamma=1}^\infty\lambda_1^{\gamma-1}\bigg(\frac{\E^{1-z}\sinh(h\gamma)}{2}+\E^{-h\gamma}G_2(z-1)+\frac{1}{2}(\E^{-h\gamma}-1-h\gamma)\bigg)\\
&&
+\lambda_1^{N+1}\bigg(\frac{\E^{-z}\lambda_1\sinh(h)}{2(\lambda_1^2+1-2\lambda_1\cosh(h))}+
\frac{\lambda_1 \E^h G_2(z)}{1-\E^h\lambda_1}-\frac{\lambda_1 h}{2(1-\lambda_1)^2}\bigg)\bigg)\bigg].
\end{eqnarray*}
\end{corollary}

\section{Numerical results}

In this section we give some numerical results.

First, when $N=5$ using Theorem \ref{THM5.1}, Corollary \ref{Cor1} and Theorem \ref{THM5} we get the graphs of the coefficients of the optimal interpolation formulas
$$
\varphi(z)\cong \mathring{P}_{\varphi}(z)=\sum\limits_{\beta=0}^5\mathring{C}_{\beta}(z)\varphi(h\beta),\ z\in [0,1]
$$
for the cases $m=1$, $m=2$ and $m=3$, respectively. They are presented in Figures \ref{Fig1}, \ref{Fig4} and \ref{Fig7}, respectively.
These graphical results confirm Remark \ref{Rem2} for the cases $N=5$ and $m=1,2,3$, i.e. for the optimal coefficients the following hold
$$
\mathring{C}_{\beta}(h\gamma)=\delta_{\beta\gamma},\ \ \beta,\gamma=0,1,...,5,
$$
where $\delta_{\beta\gamma}$ is the Kronecker symbol.

Now, in numerical examples, we interpolate the functions
$$
\varphi_1(z)=z^2,\ \varphi_2(z)=\E^{2z}\mbox{ and } \varphi_3(z)=\sin z
$$
by optimal interpolation formulas of the form (\ref{(1)}) in the cases $m=1,\ 2,\ 3$ and $N=5,\ 10$, using Theorem \ref{THM5.1}, Corollary \ref{Cor1} and Theorem \ref{THM5}. For the functions
$\varphi_i$, $i=1,2,3$ the graphs of absolute errors $|\varphi_i(z)-\mathring{P}_{\varphi_i}(z)|$, $i=1,2,3$, are given in Figures \ref{Fig2}, \ref{Fig3}, \ref{Fig5}, \ref{Fig6}, \ref{Fig8} and \ref{Fig9}. In these Figures one can see that by increasing values of $m$ and $N$ absolute errors
between optimal interpolation formulas and given functions are decreasing.

\section*{Acknowledgements} We are very thankful to professor Dario Andrea Bini for discussion of the results
of this paper. S.S. Babaev thanks professor Dario Andrea Bini and his research group for hospitality.
The part of this work was done at the Pisa University, Italy. The first
author thanks the ERASMUS+ KA107 International Credit Mobility for scholarship.

\newpage

\begin{figure}
  \includegraphics[width=0.3\textwidth]{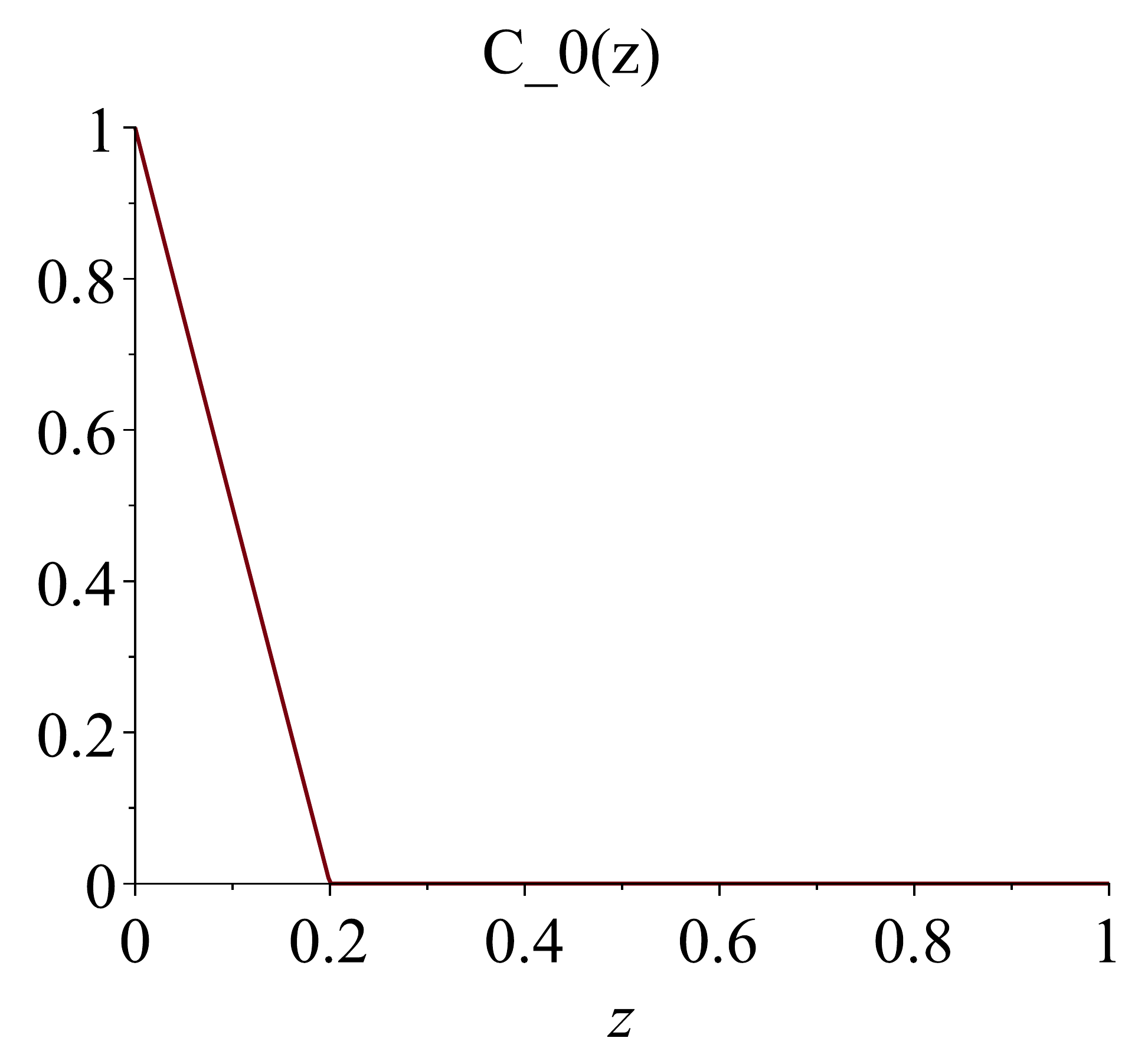}
  \includegraphics[width=0.3\textwidth]{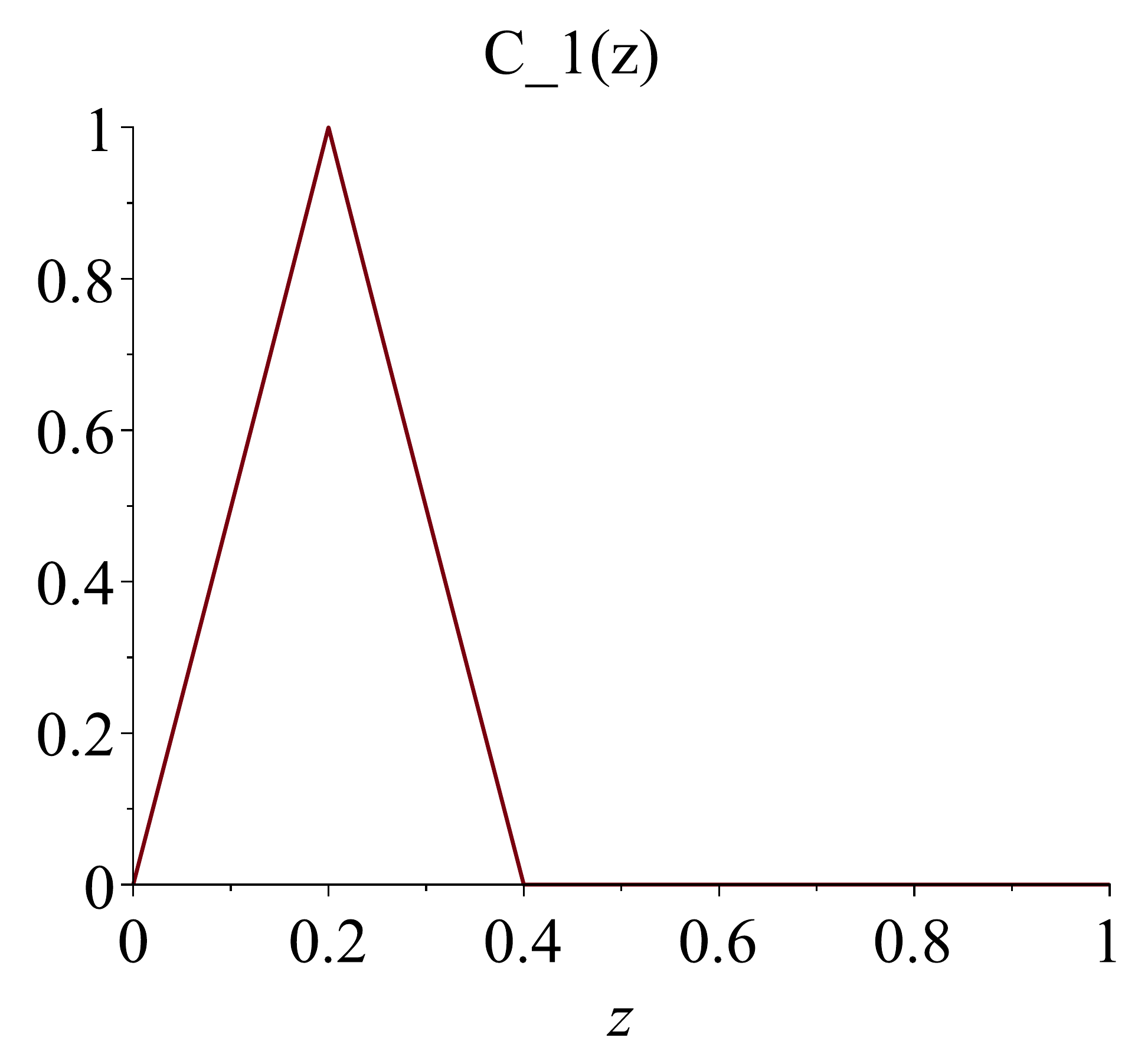}
  \includegraphics[width=0.3\textwidth]{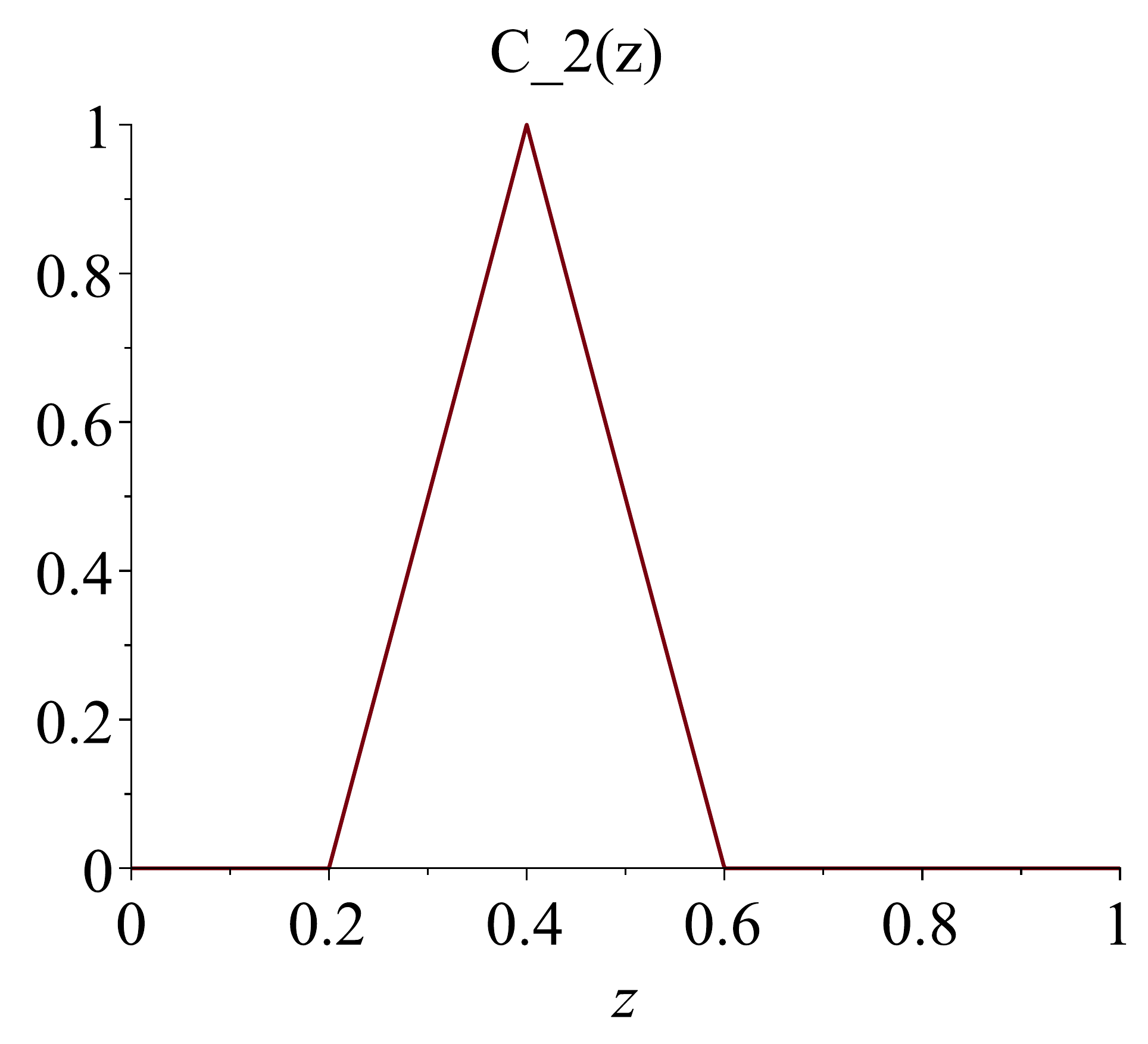}\\
  \includegraphics[width=0.3\textwidth]{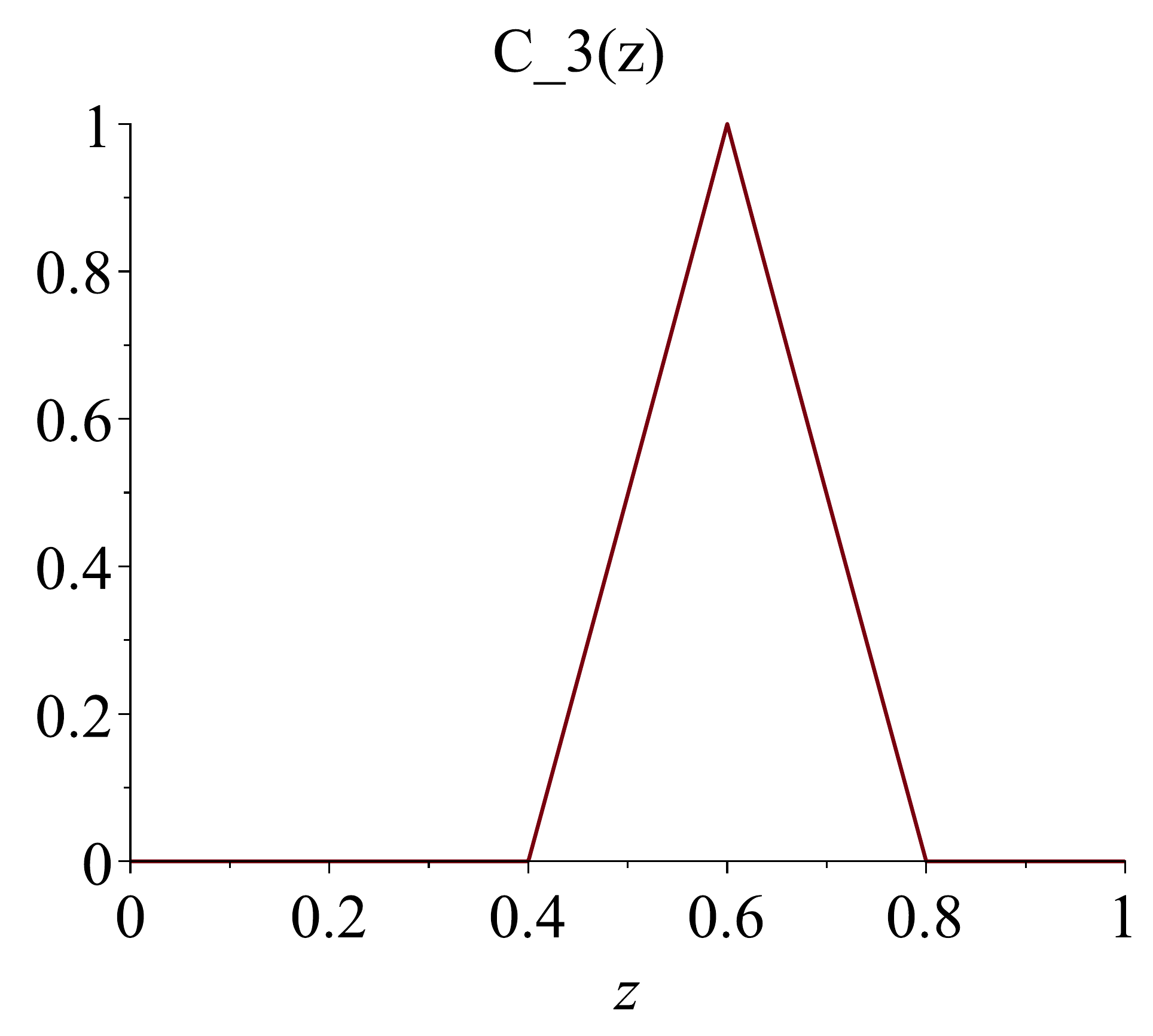}
  \includegraphics[width=0.3\textwidth]{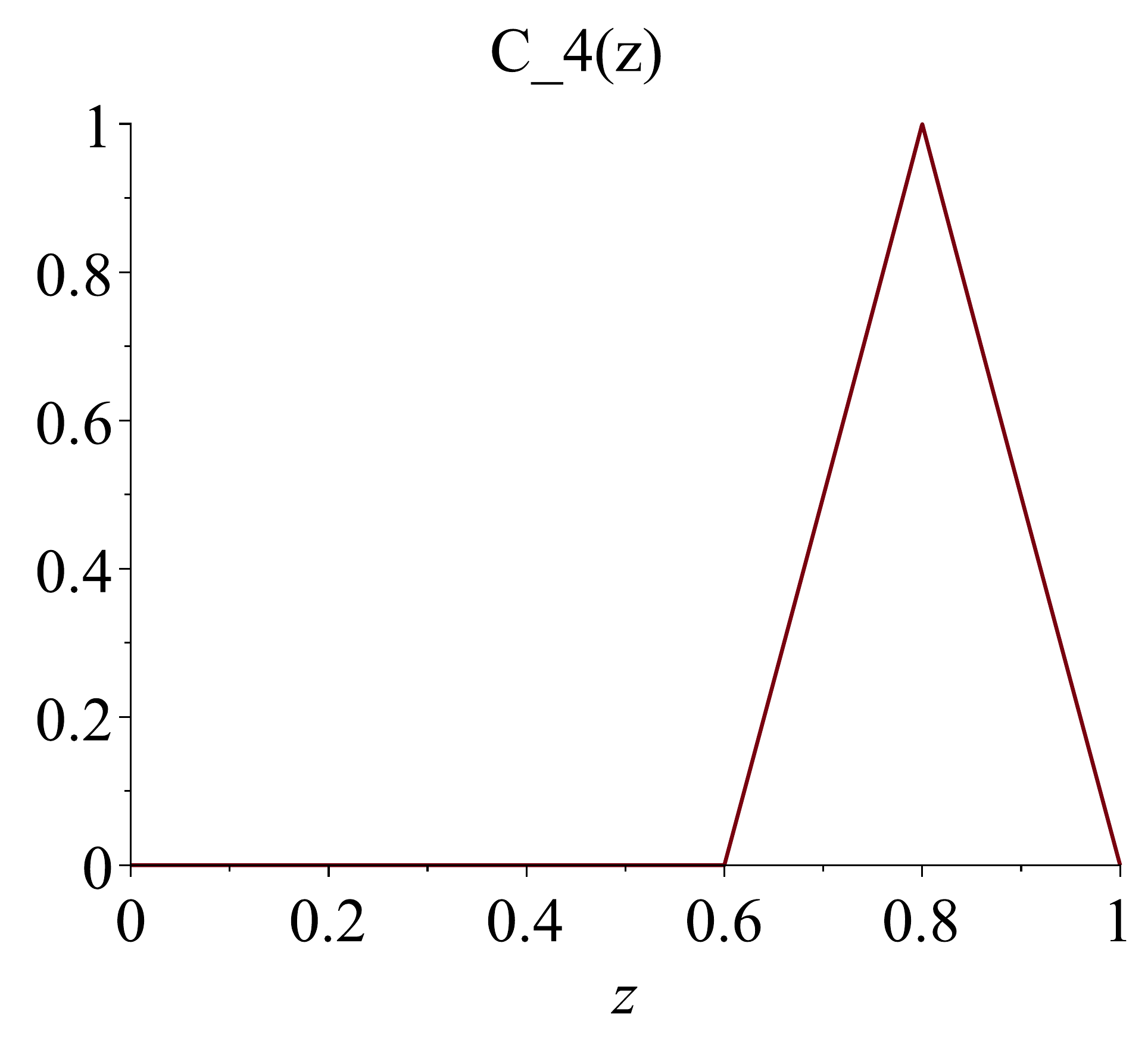}
  \includegraphics[width=0.3\textwidth]{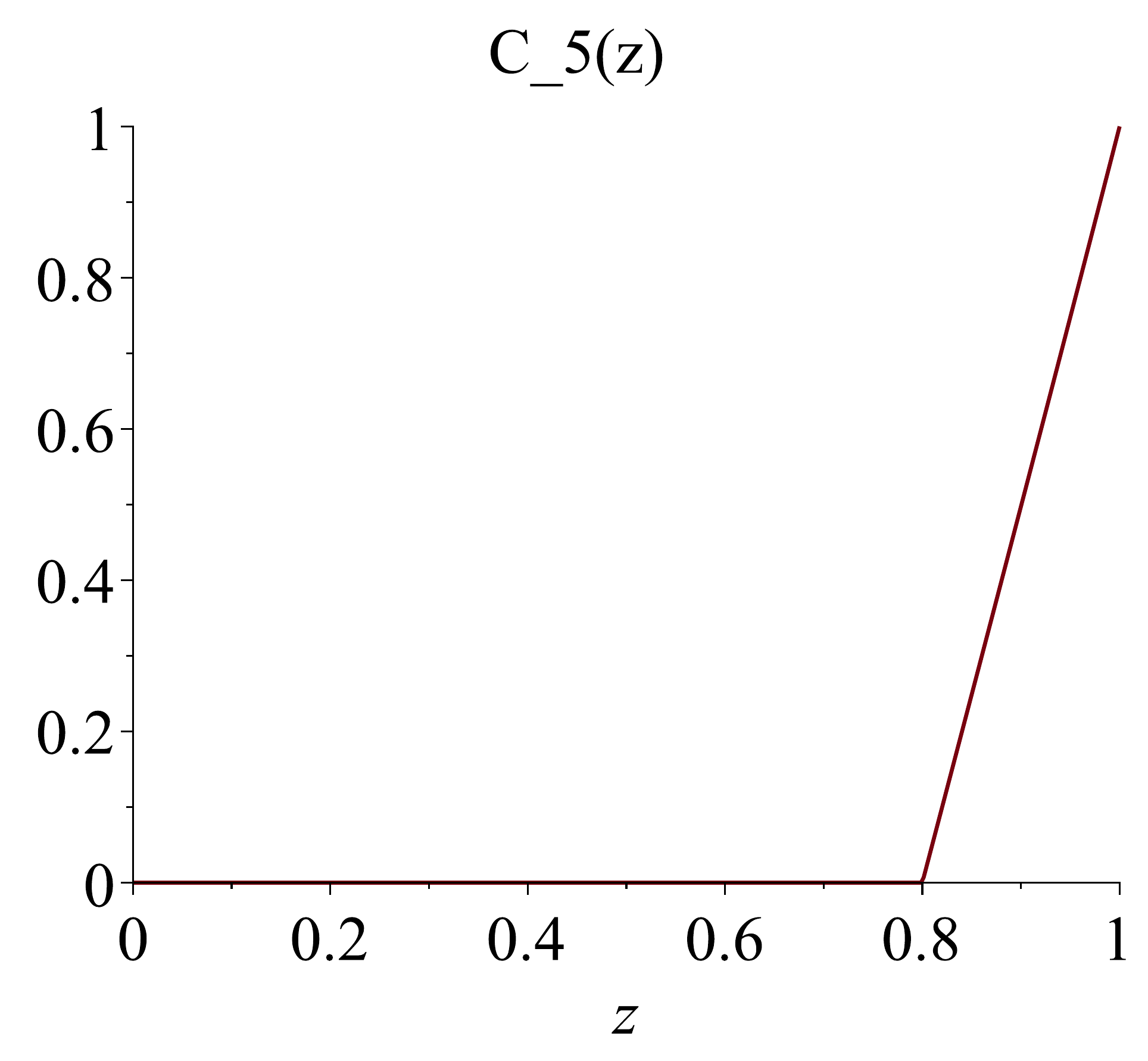}
  \caption{Graphs of coefficients of the optimal interpolation formulas (\ref{(1)}) in the case $m=1$ and $N=5$.}\label{Fig1}
\end{figure}

\begin{figure}
  \includegraphics[width=0.3\textwidth]{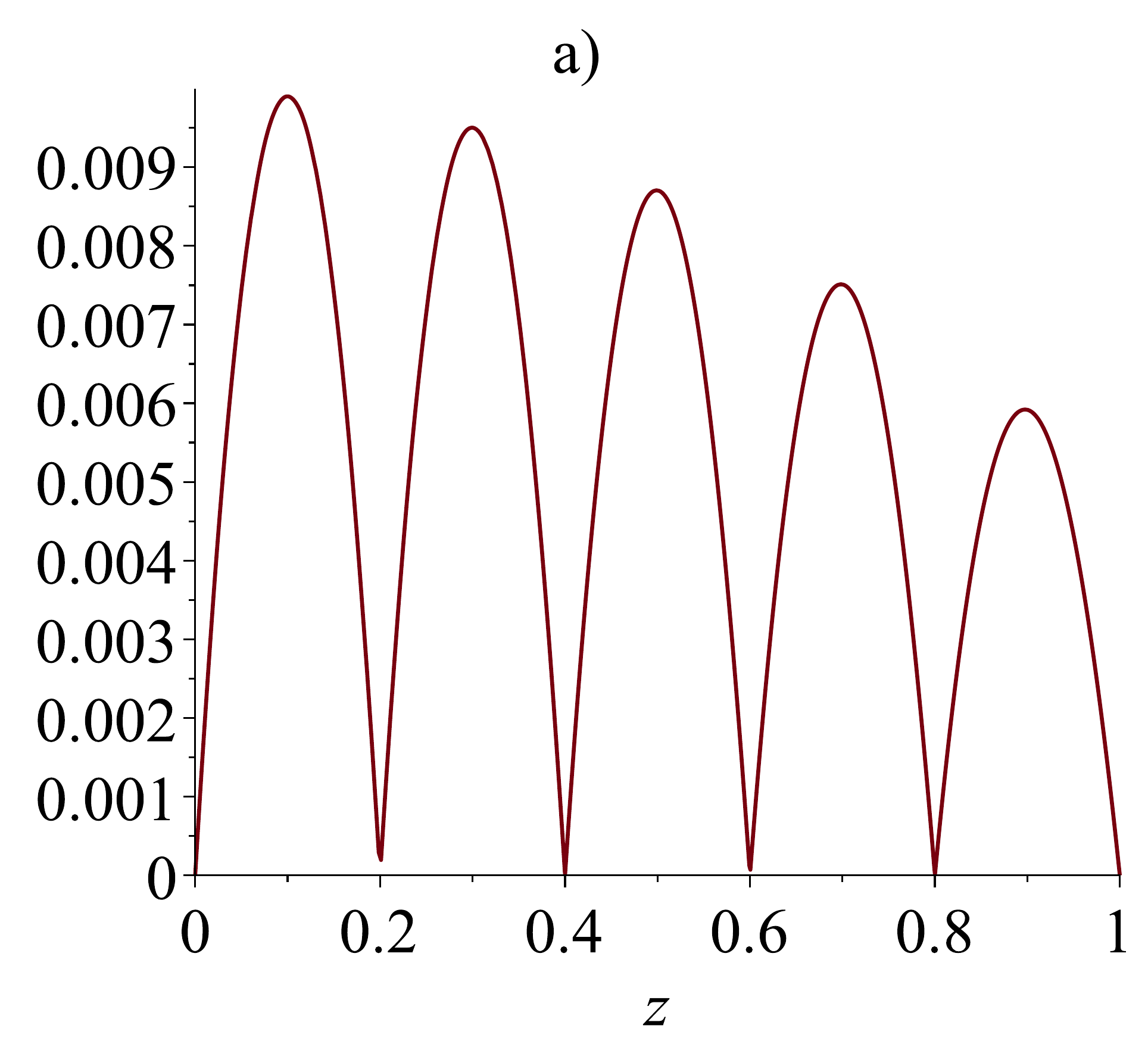}
  \includegraphics[width=0.3\textwidth]{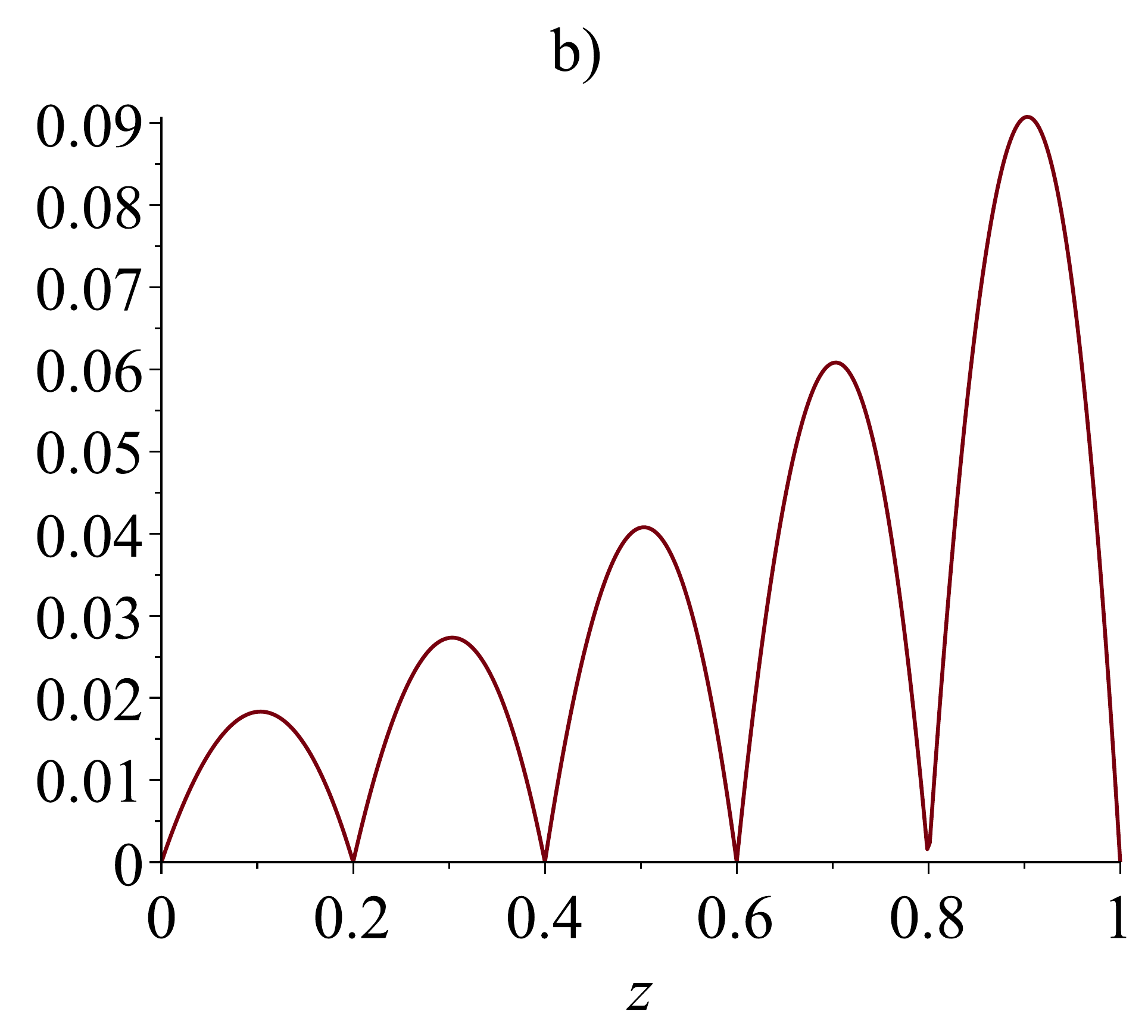}
  \includegraphics[width=0.3\textwidth]{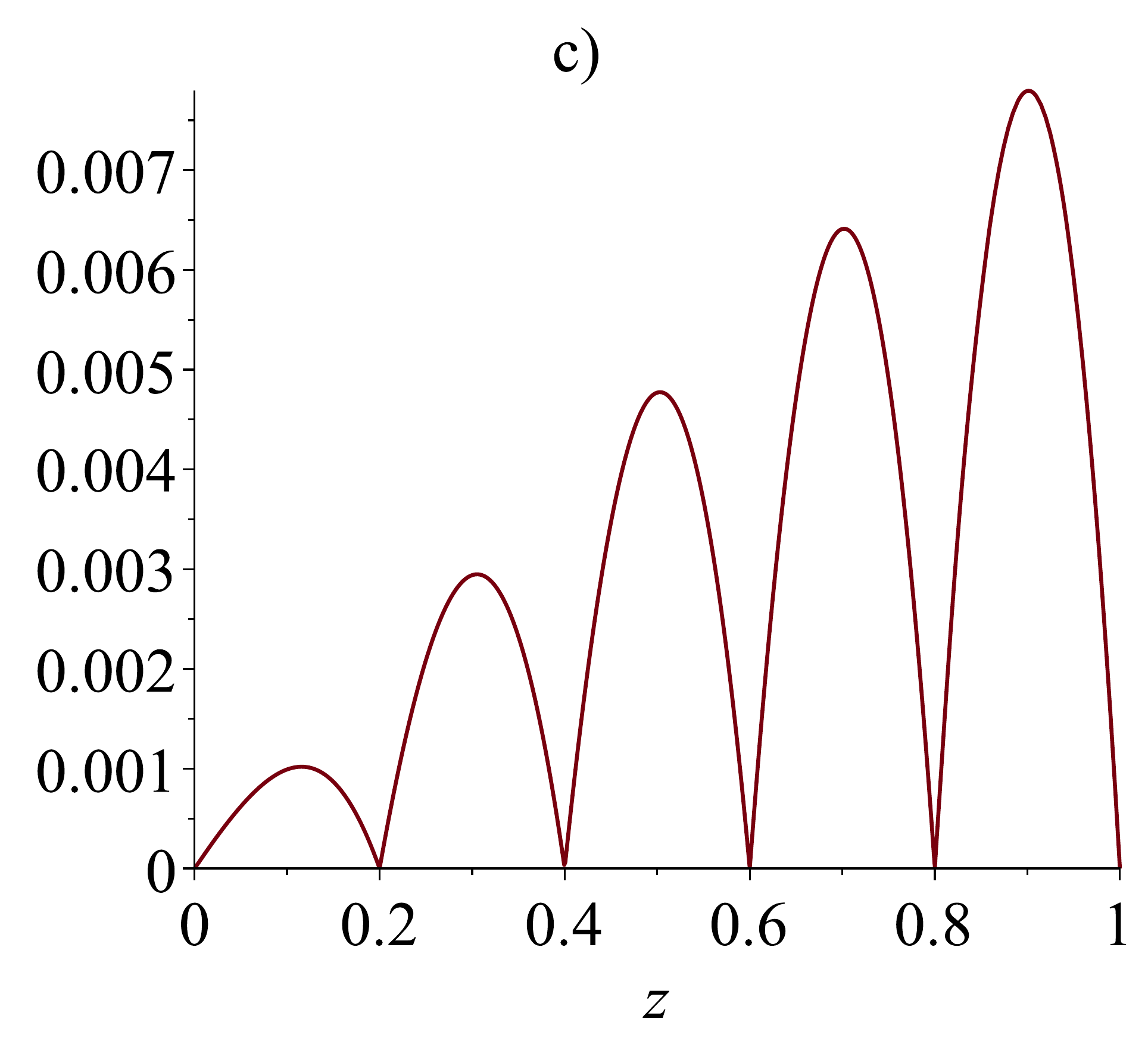}
  \caption{Graphs of absolute errors for $m=1$ and $N=5$: (\textbf{a}) $|\varphi_1(z)-P_{\varphi_1}(z)|$, (\textbf{b}) $|\varphi_2(z)-P_{\varphi_2}(z)|$, (\textbf{c}) $|\varphi_3(z)-P_{\varphi_3}(z)|$.}\label{Fig2}
\end{figure}

\begin{figure}
  \includegraphics[width=0.3\textwidth]{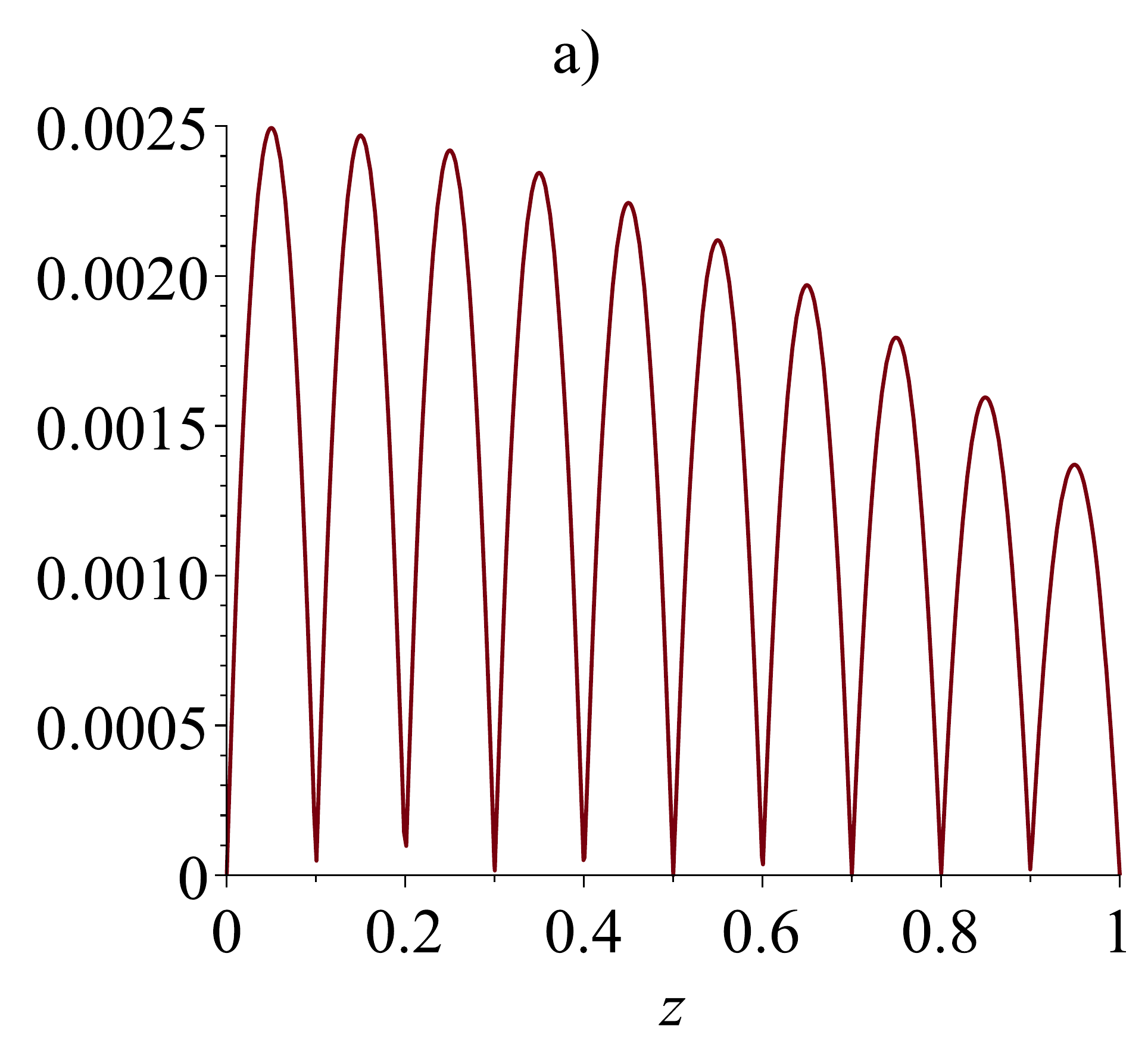}
  \includegraphics[width=0.3\textwidth]{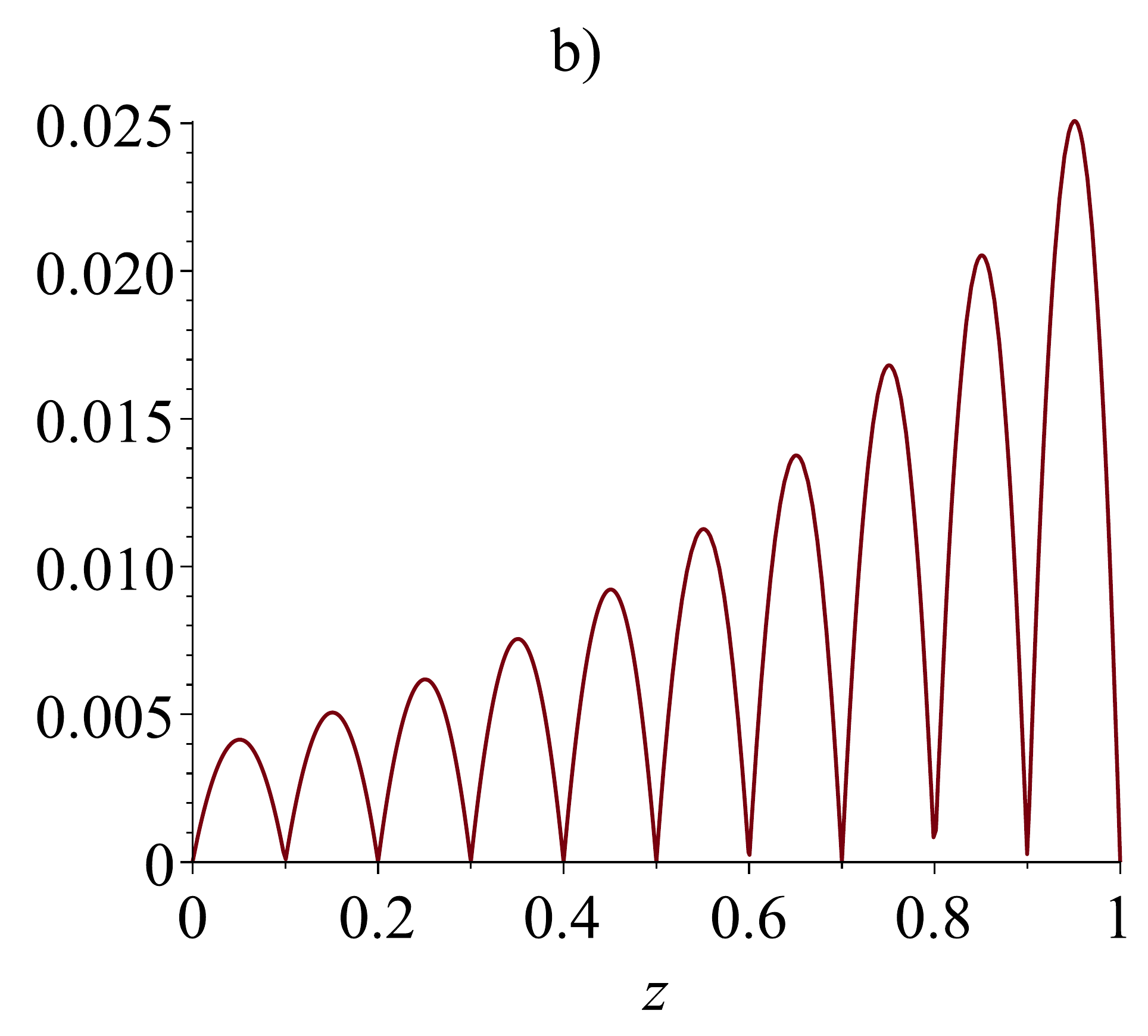}
  \includegraphics[width=0.3\textwidth]{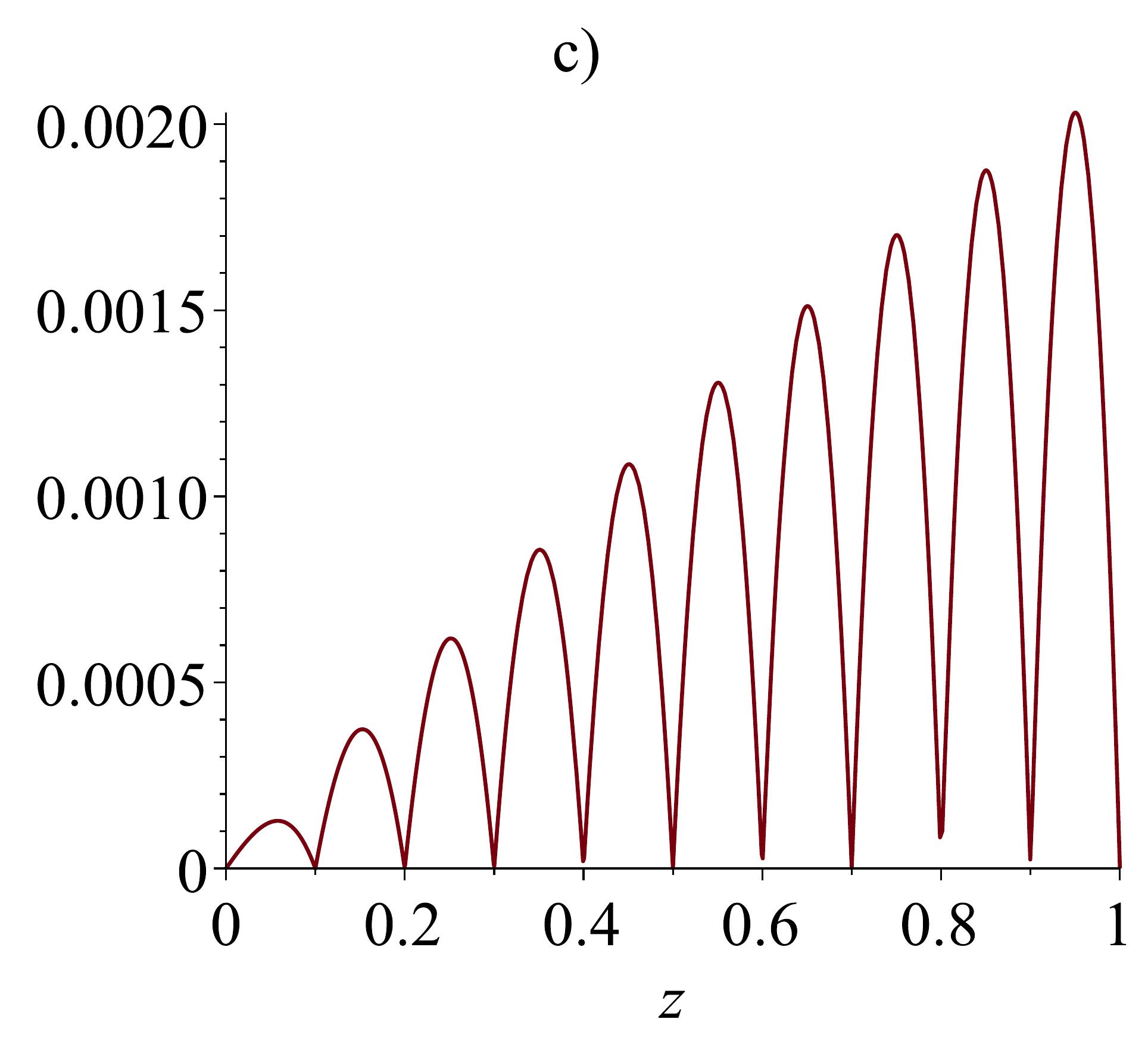}
  \caption{Graphs of absolute errors for $m=1$ and $N=10$: (\textbf{a}) $|\varphi_1(z)-P_{\varphi_1}(z)|$, (\textbf{b}) $|\varphi_2(z)-P_{\varphi_2}(z)|$, (\textbf{c}) $|\varphi_3(z)-P_{\varphi_3}(z)|$.}\label{Fig3}
\end{figure}

\begin{figure}
  \includegraphics[width=0.3\textwidth]{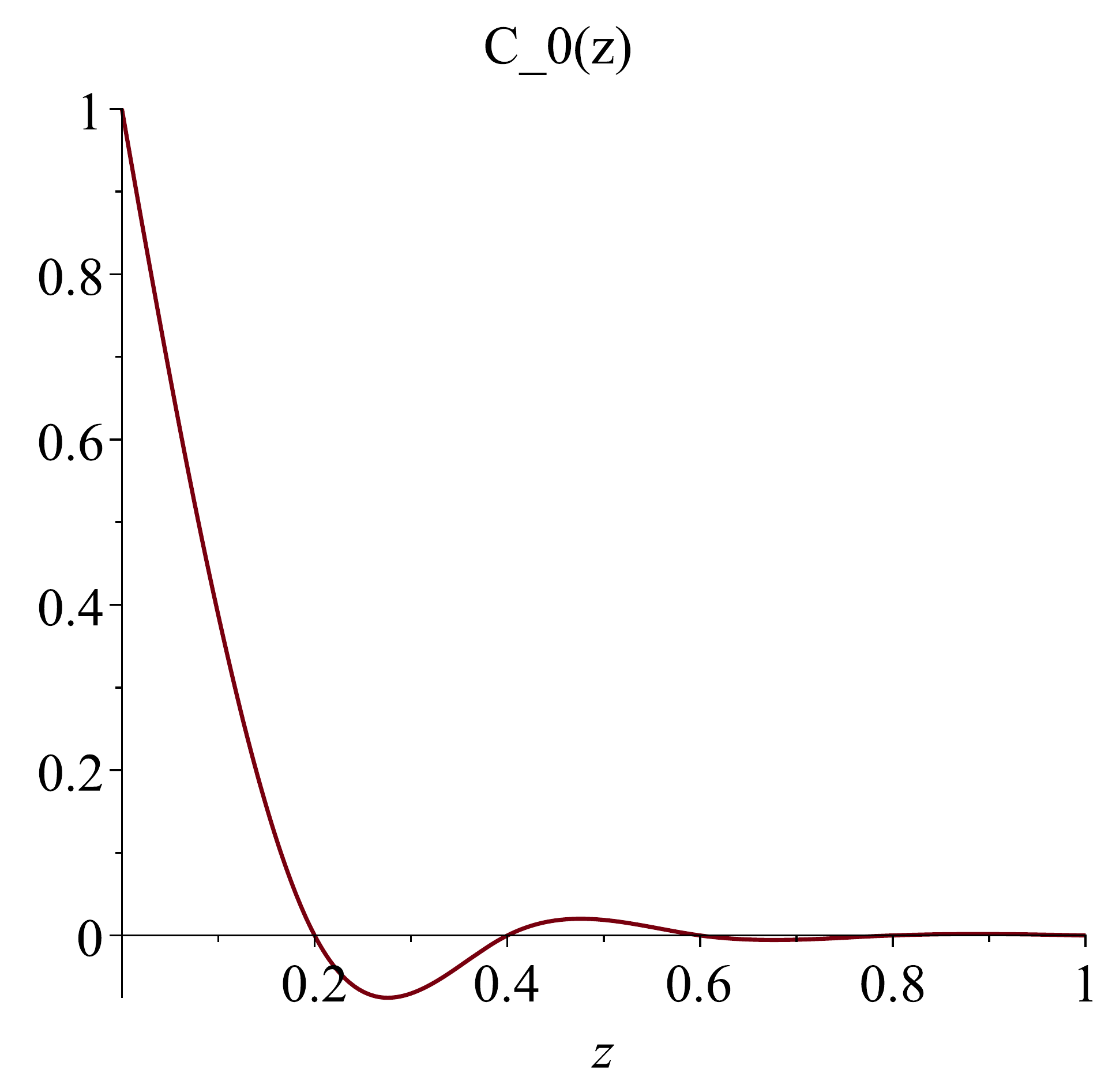}
  \includegraphics[width=0.3\textwidth]{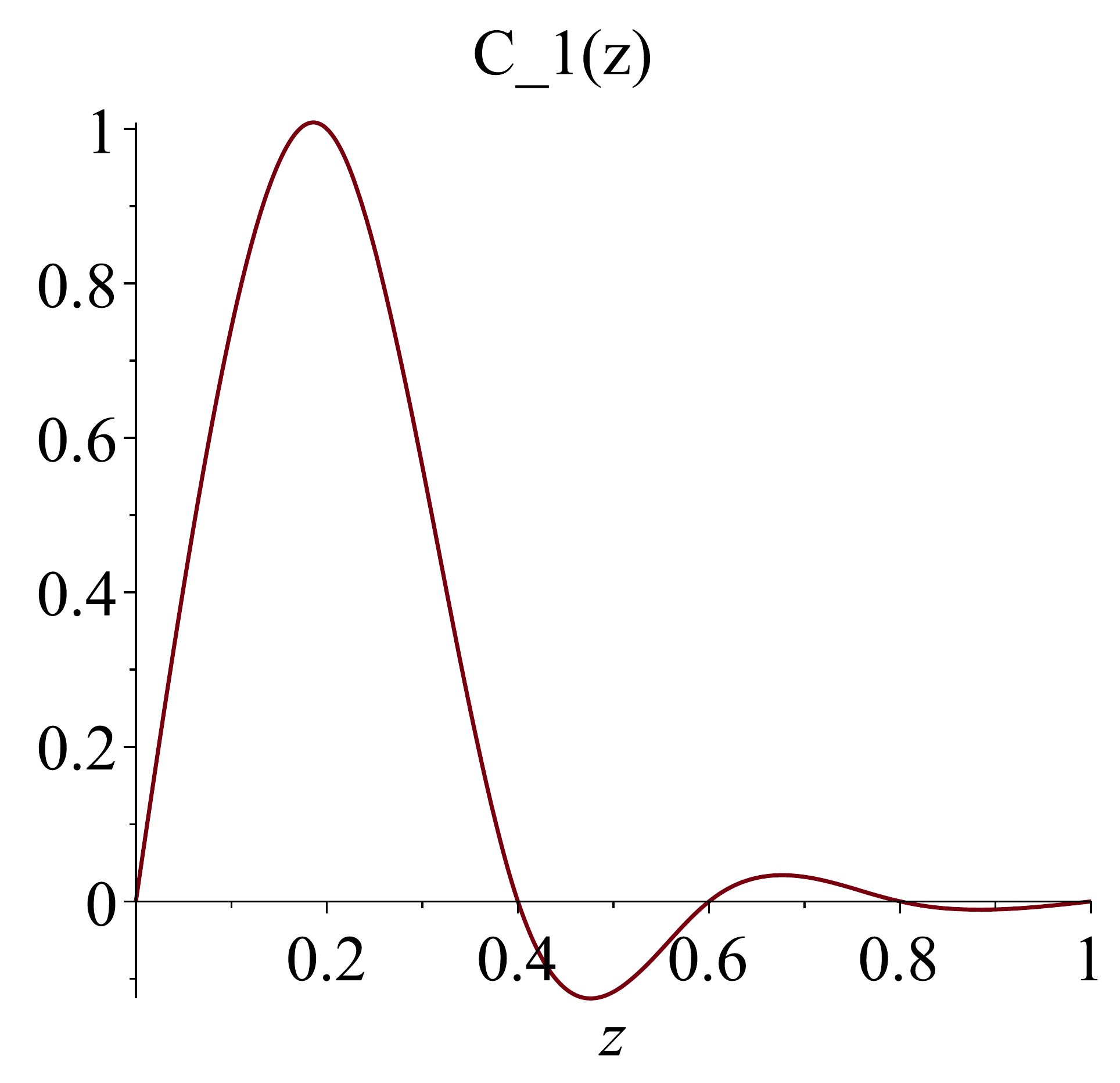}
  \includegraphics[width=0.3\textwidth]{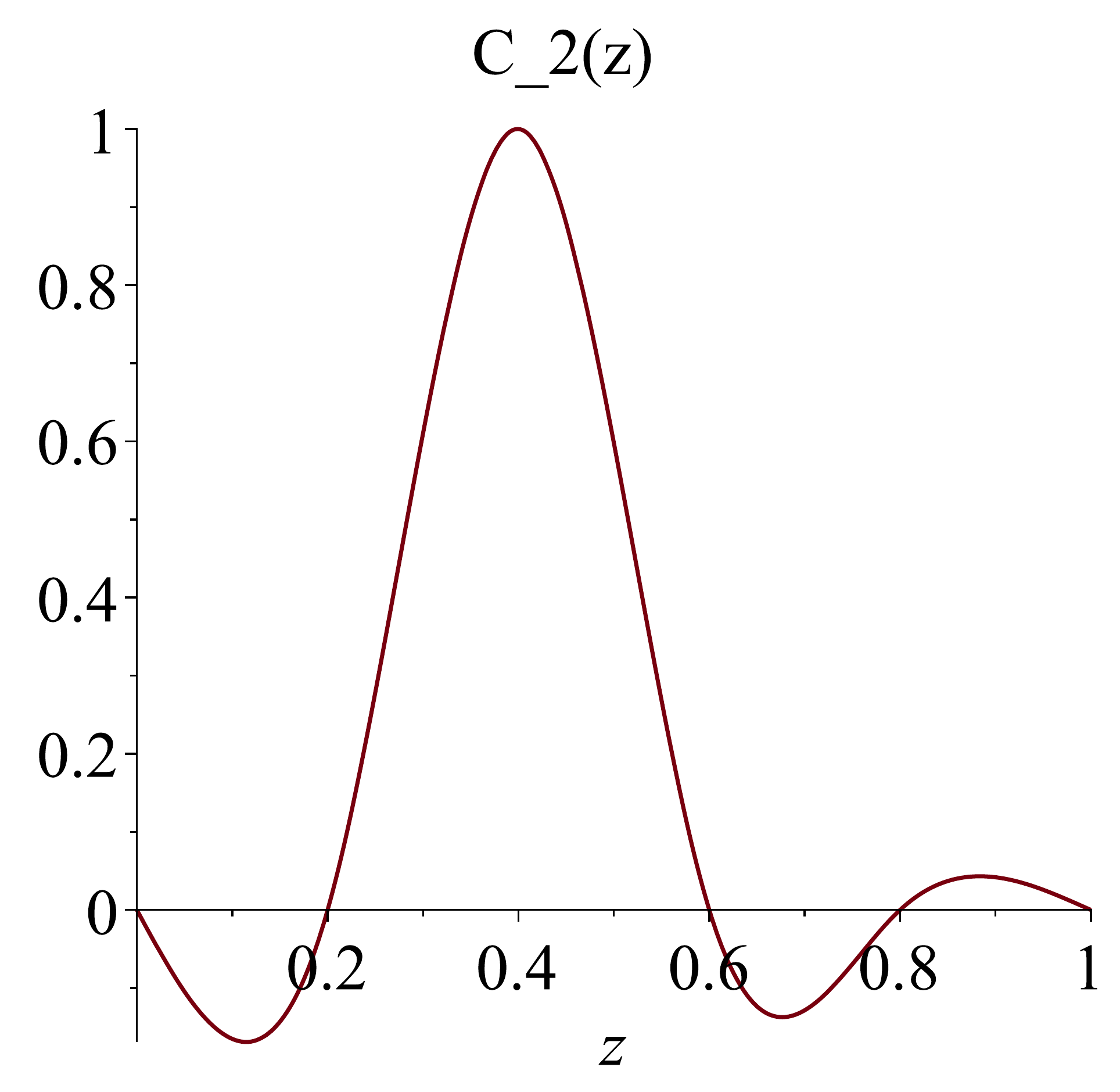}\\
  \includegraphics[width=0.3\textwidth]{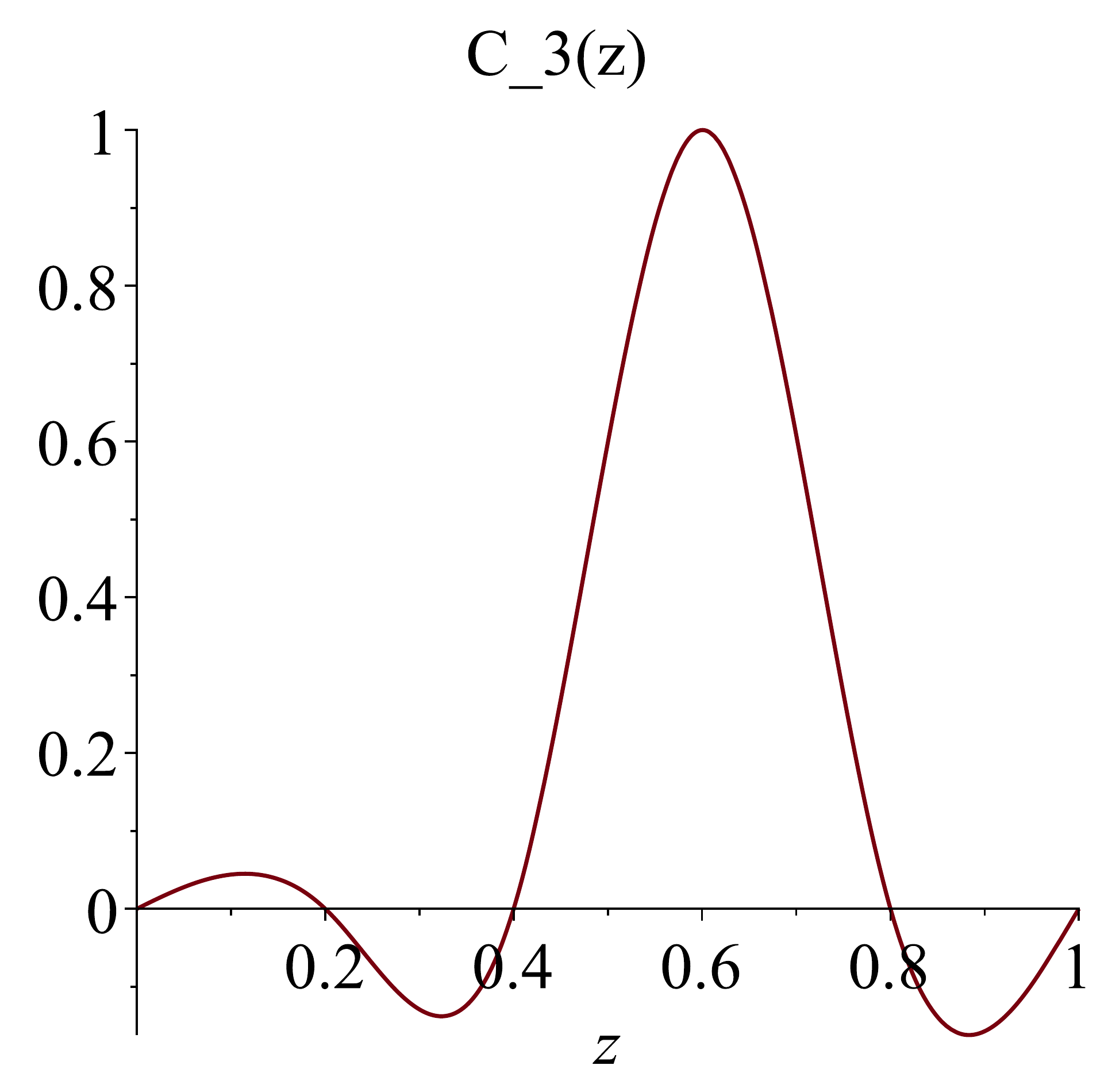}
  \includegraphics[width=0.3\textwidth]{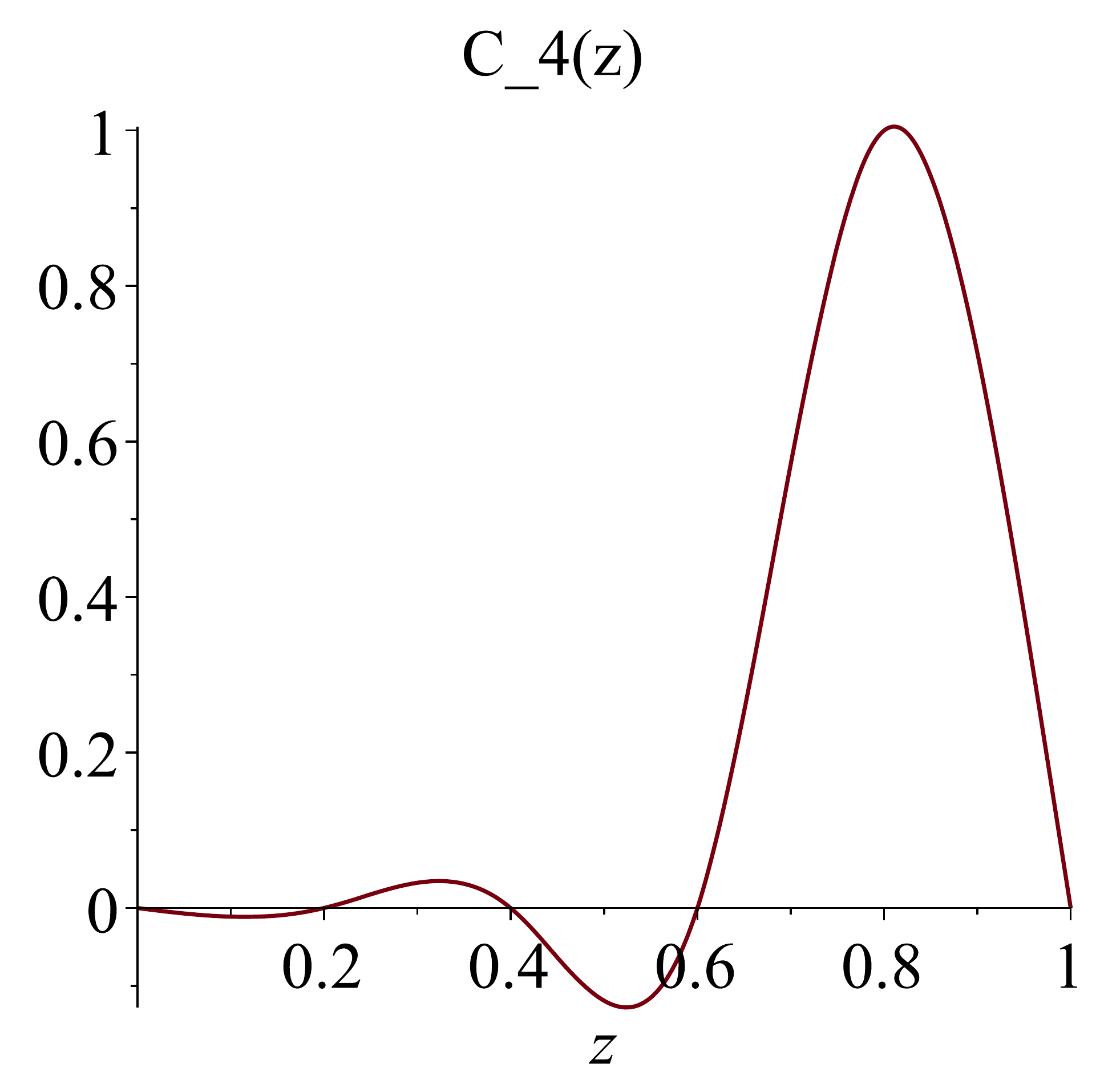}
  \includegraphics[width=0.3\textwidth]{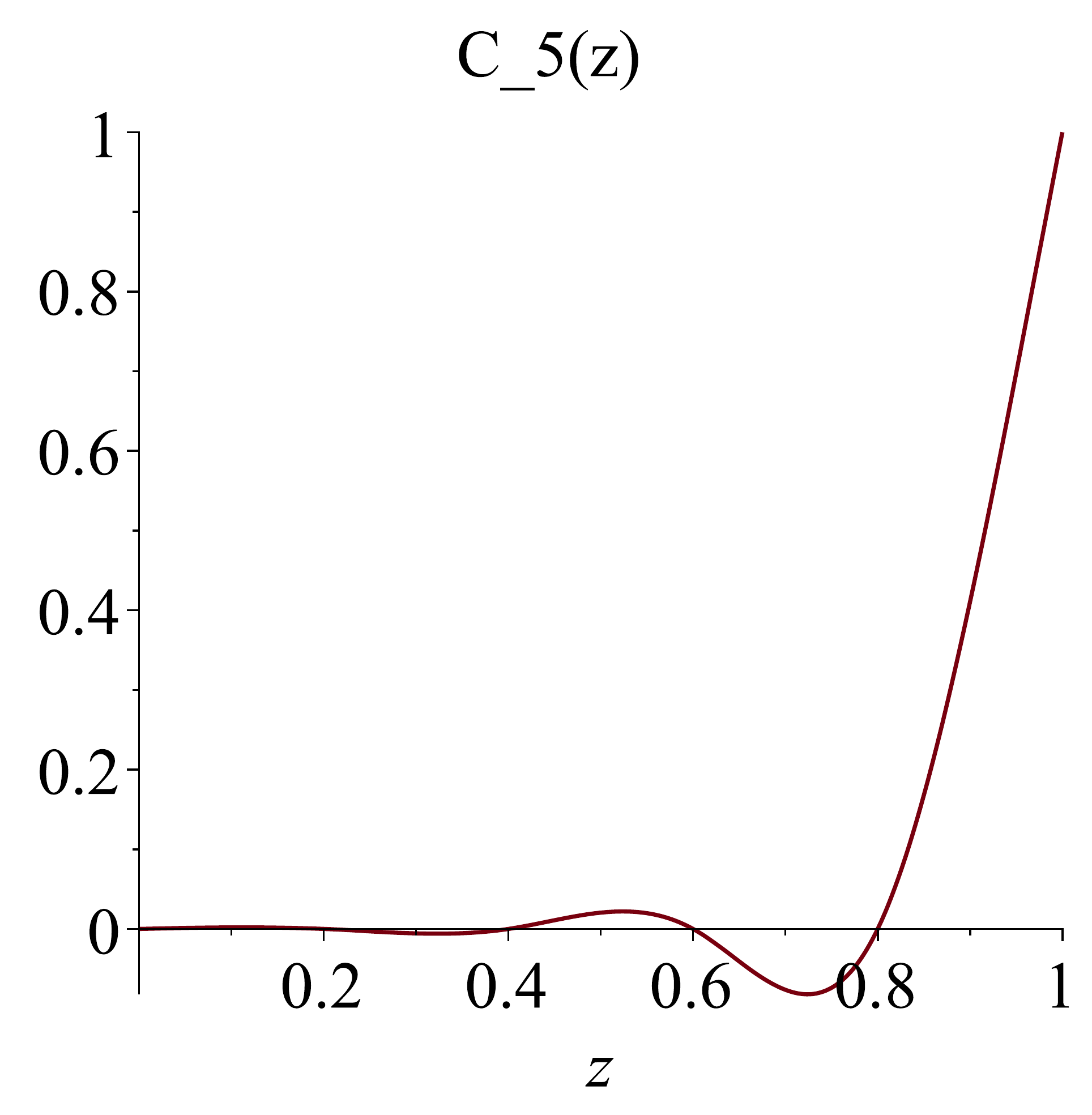}
  \caption{Graphs of coefficients of the optimal interpolation formulas (\ref{(1)}) in the case $m=2$ and $N=5$.}\label{Fig4}
\end{figure}

\begin{figure}
  \includegraphics[width=0.3\textwidth]{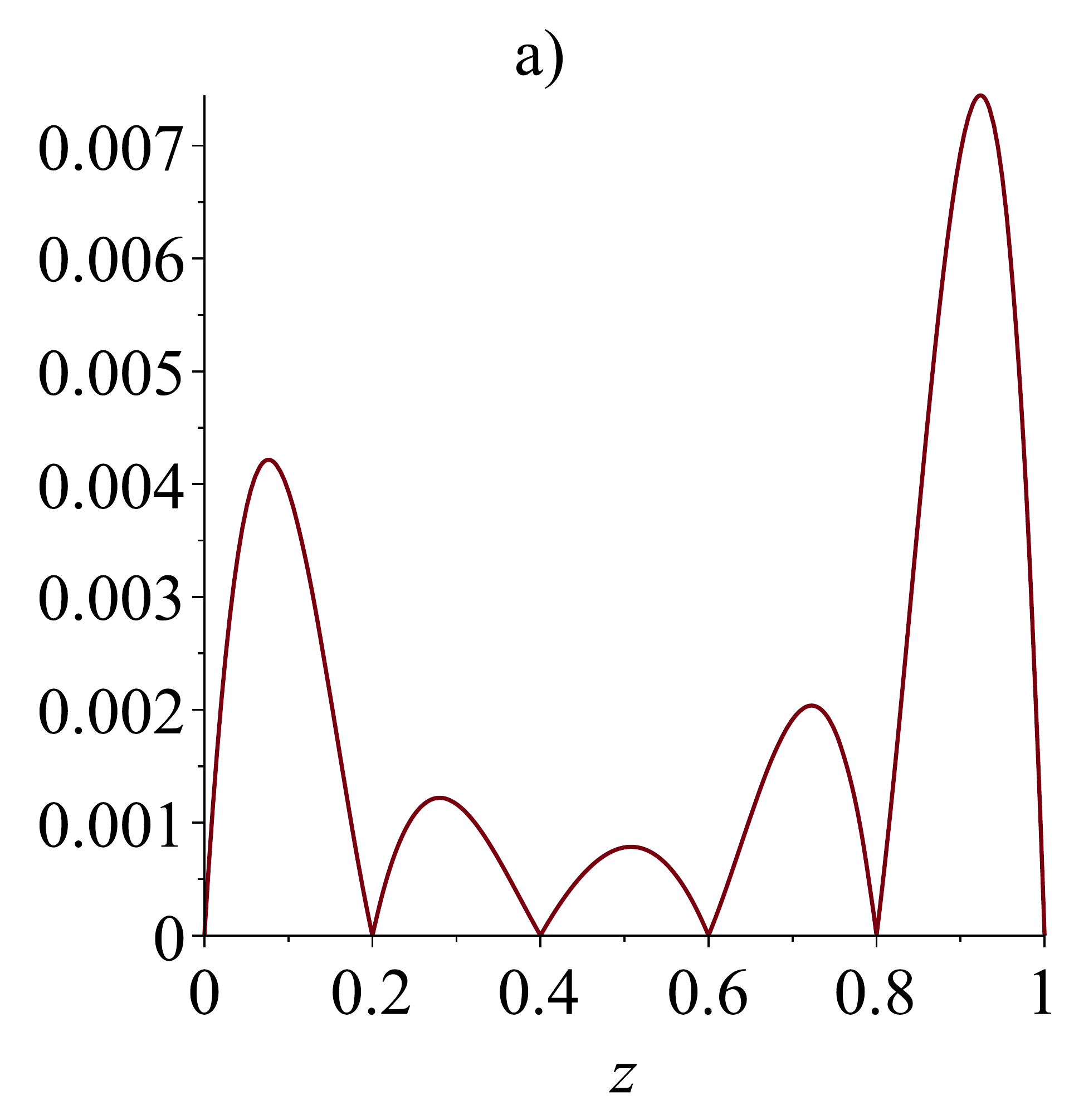}
  \includegraphics[width=0.3\textwidth]{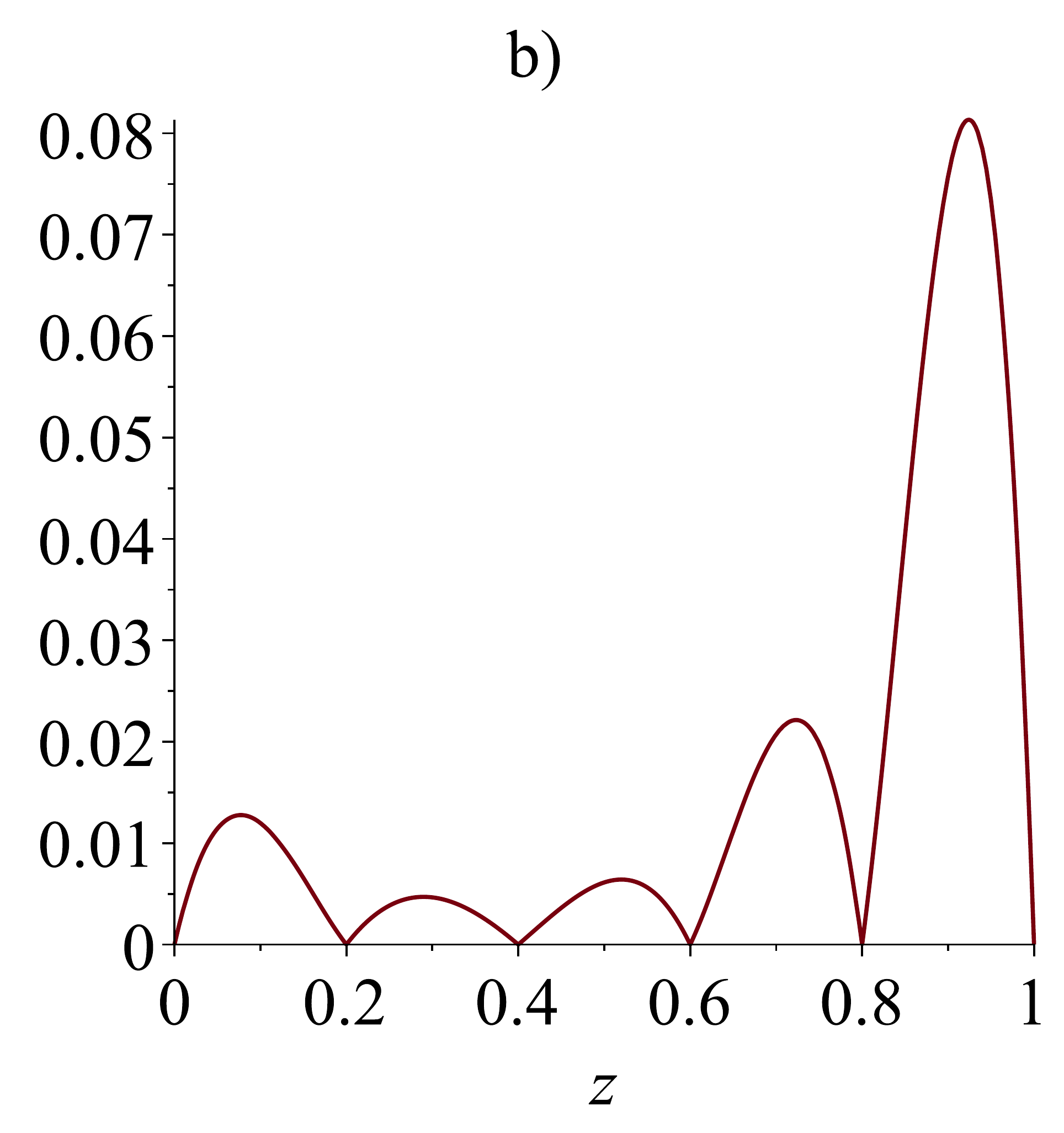}
  \includegraphics[width=0.3\textwidth]{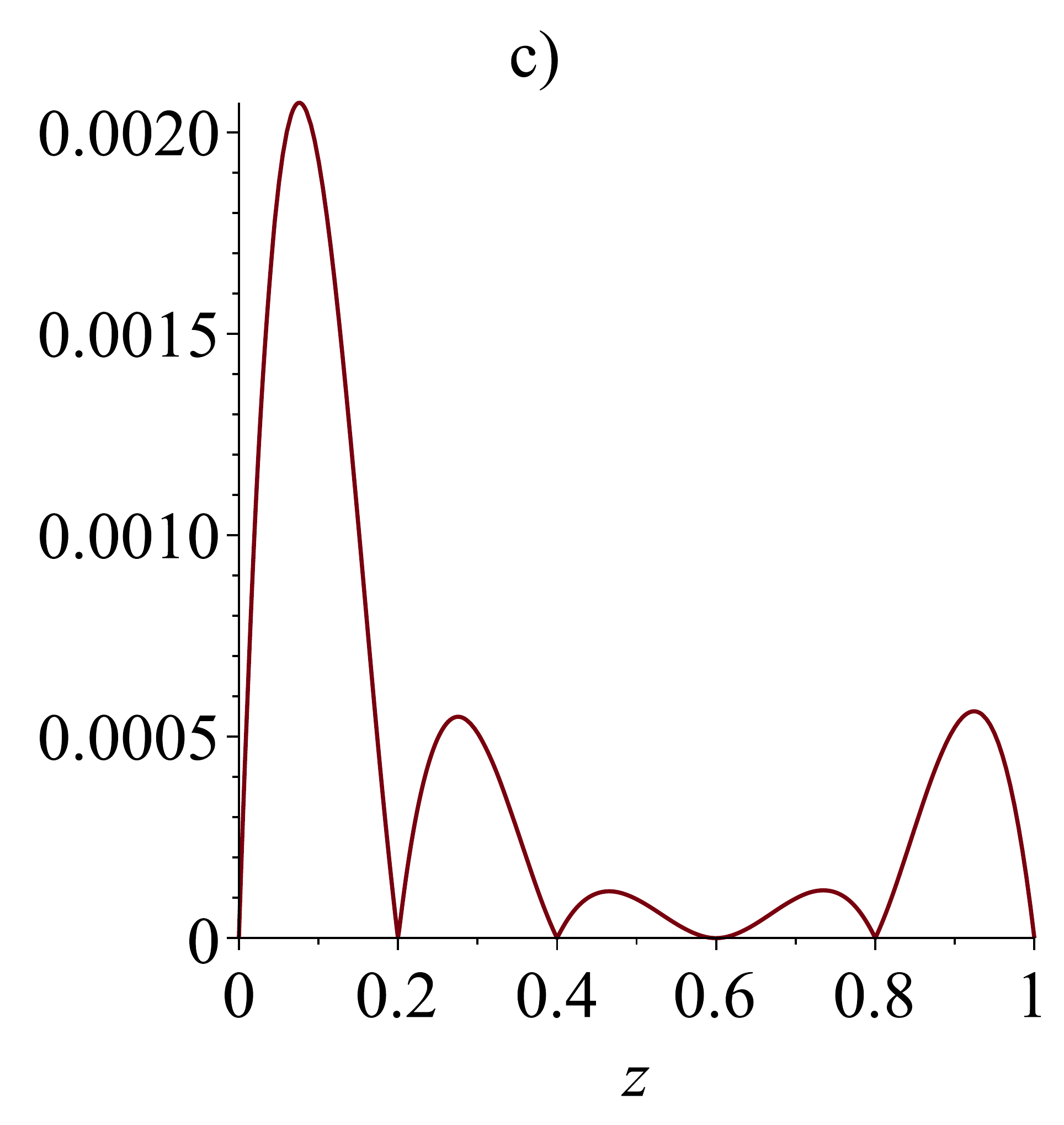}
  \caption{Graphs of absolute errors for $m=2$ and $N=5$: (\textbf{a}) $|\varphi_1(z)-P_{\varphi_1}(z)|$, (\textbf{b}) $|\varphi_2(z)-P_{\varphi_2}(z)|$, (\textbf{c}) $|\varphi_3(z)-P_{\varphi_3}(z)|$.}\label{Fig5}
\end{figure}

\begin{figure}
  \includegraphics[width=0.3\textwidth]{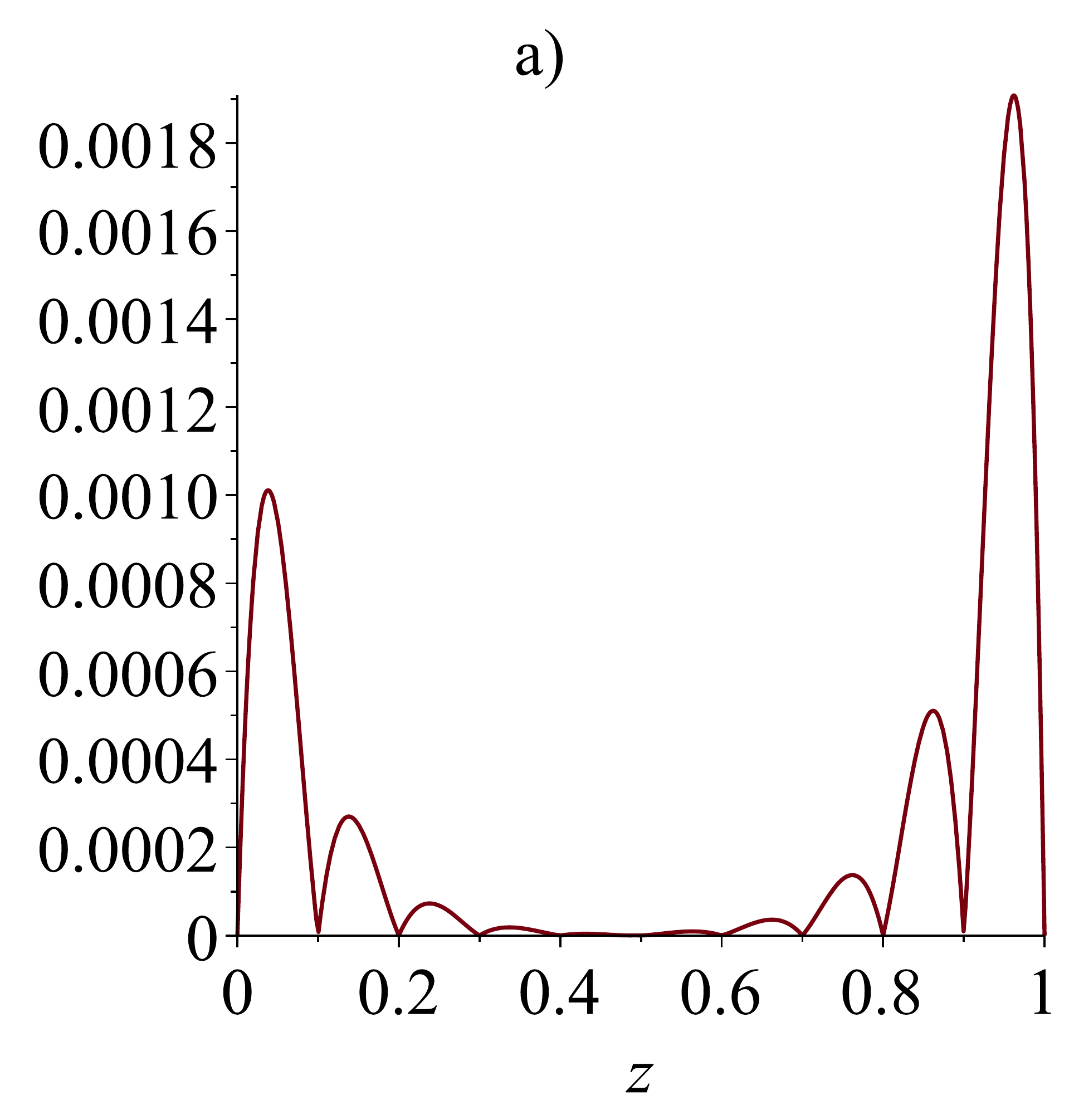}
  \includegraphics[width=0.3\textwidth]{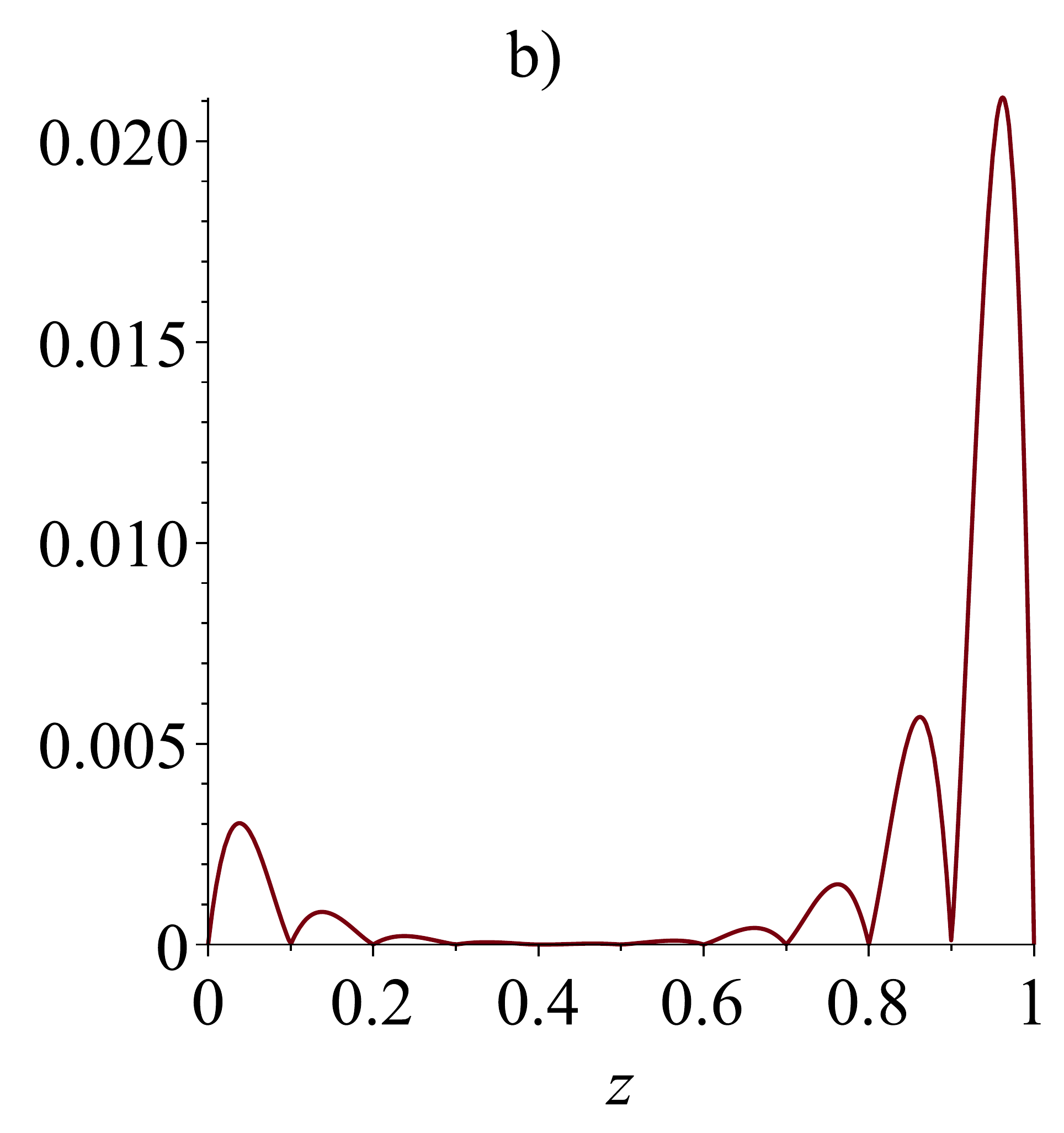}
  \includegraphics[width=0.3\textwidth]{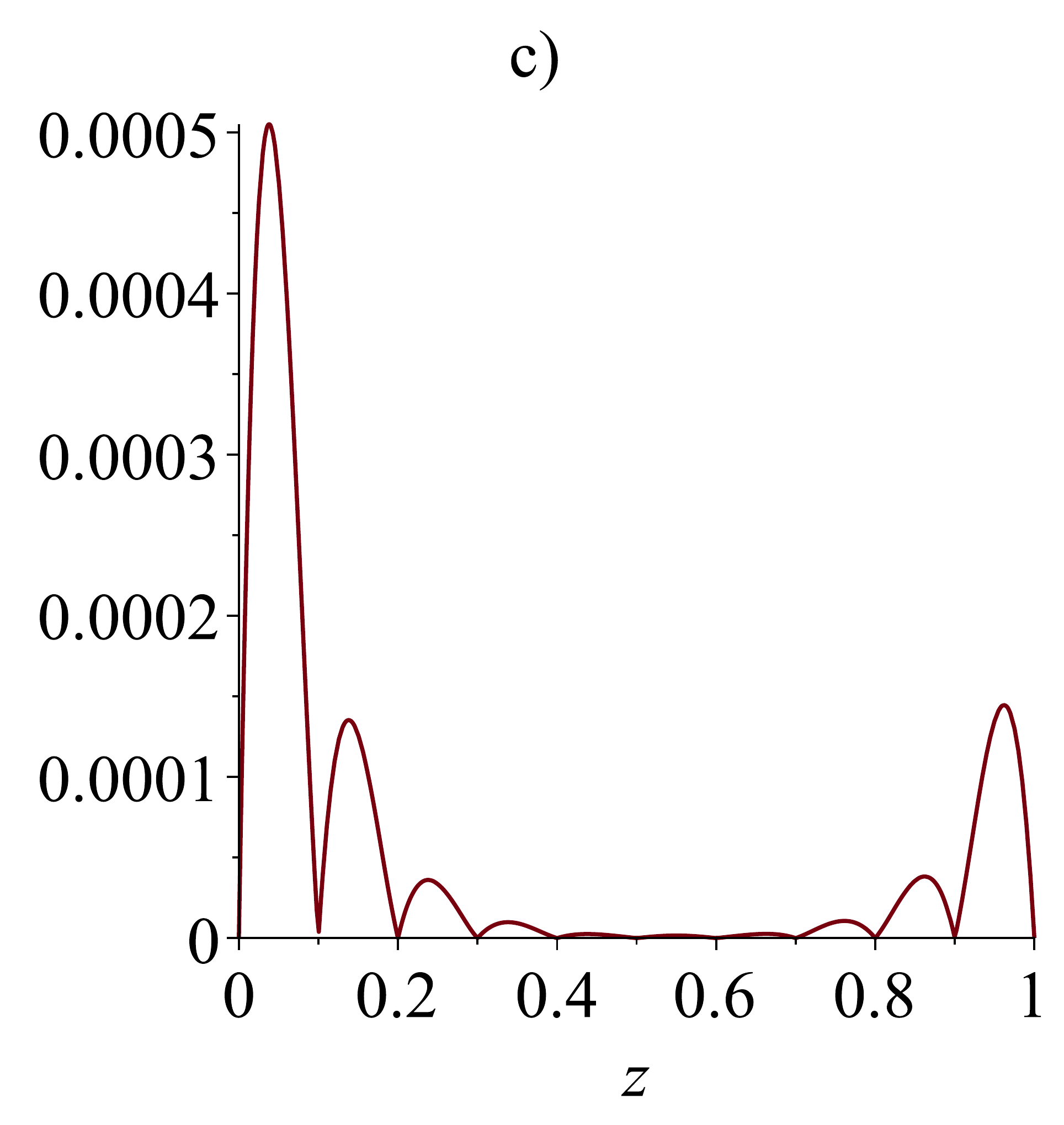}
  \caption{Graphs of absolute errors for $m=2$ and $N=10$: (\textbf{a}) $|\varphi_1(z)-P_{\varphi_1}(z)|$, (\textbf{b}) $|\varphi_2(z)-P_{\varphi_2}(z)|$, (\textbf{c}) $|\varphi_3(z)-P_{\varphi_3}(z)|$.}\label{Fig6}
\end{figure}

\begin{figure}
  \includegraphics[width=0.3\textwidth]{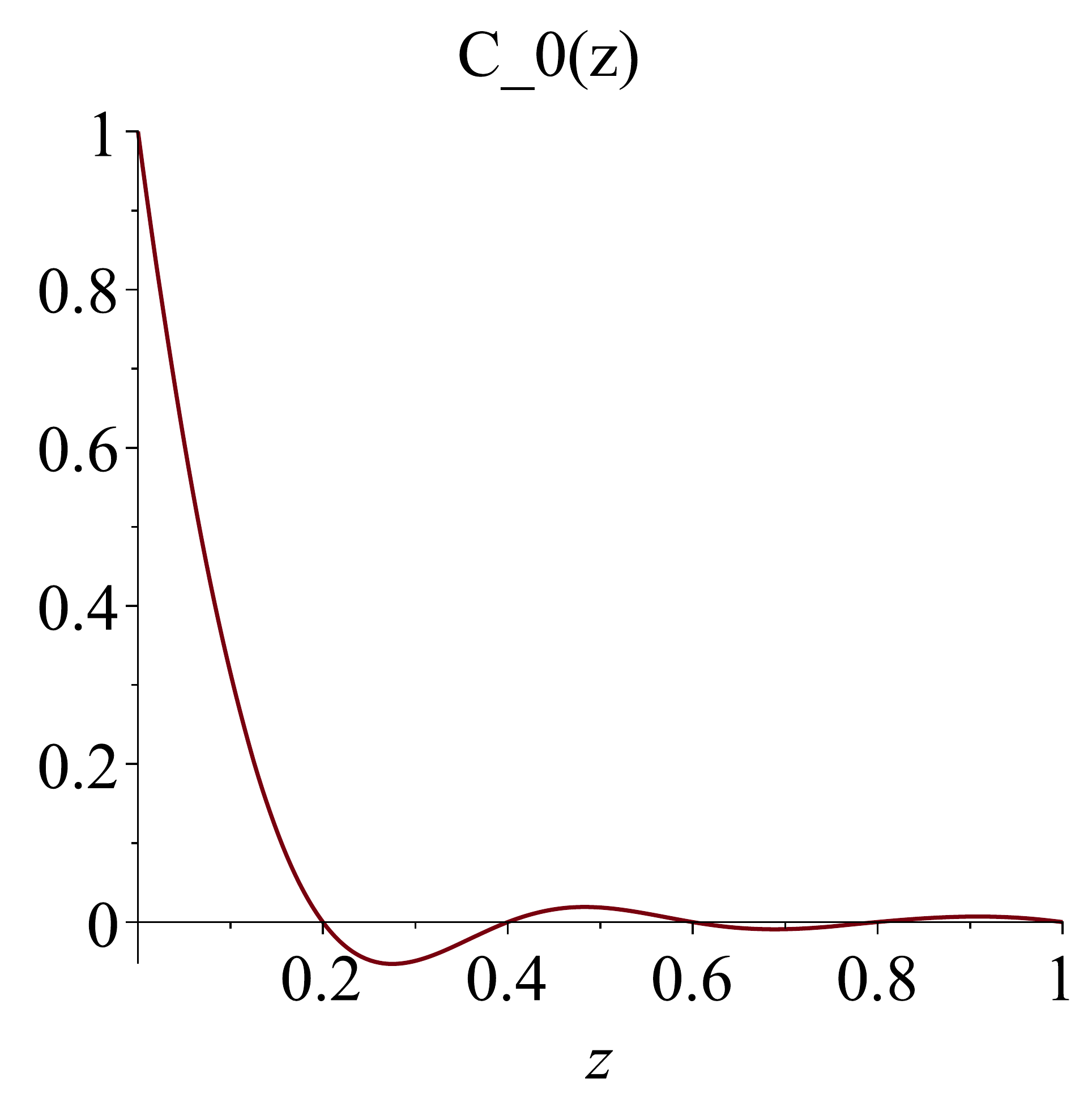}
  \includegraphics[width=0.3\textwidth]{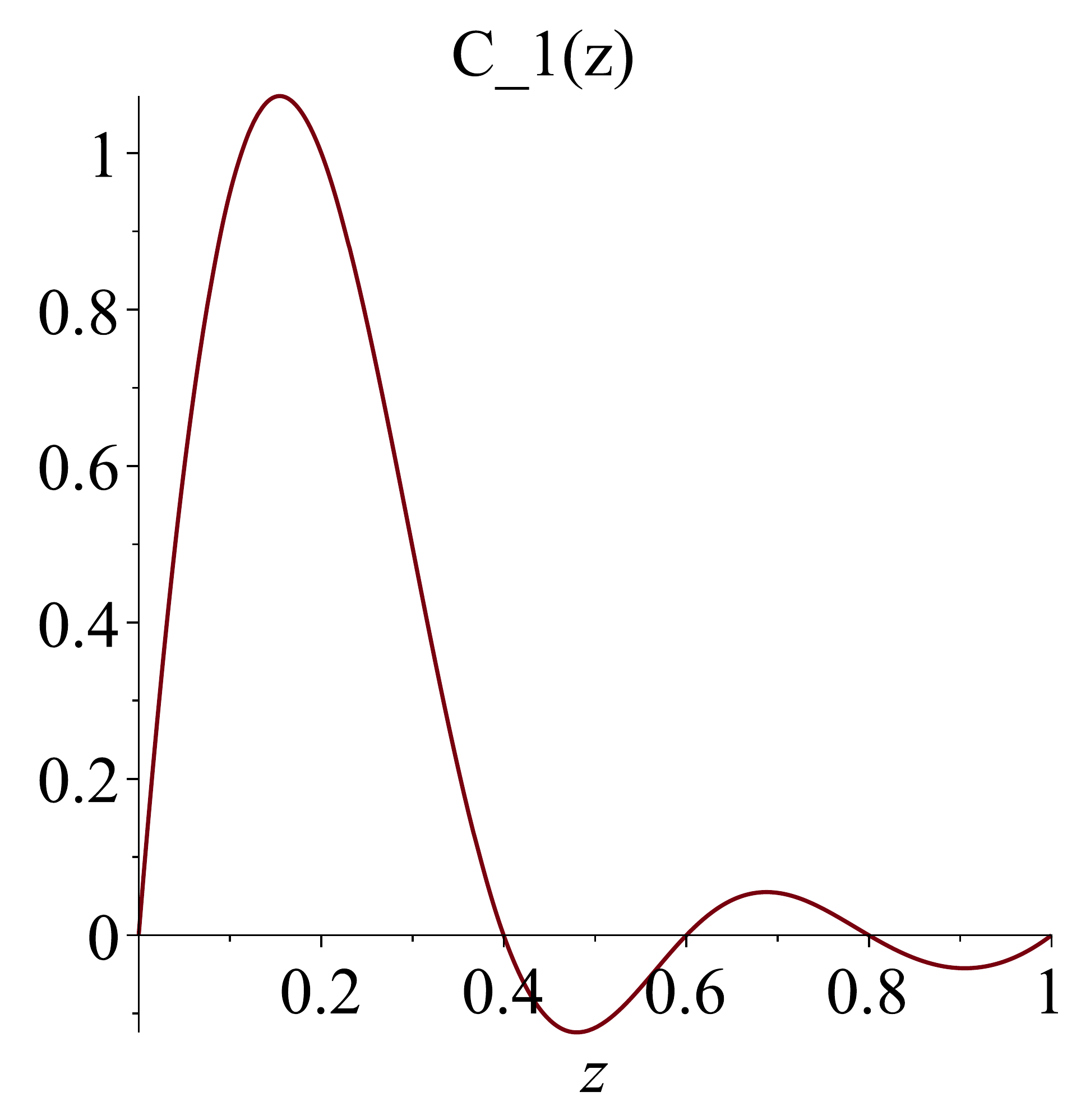}
  \includegraphics[width=0.3\textwidth]{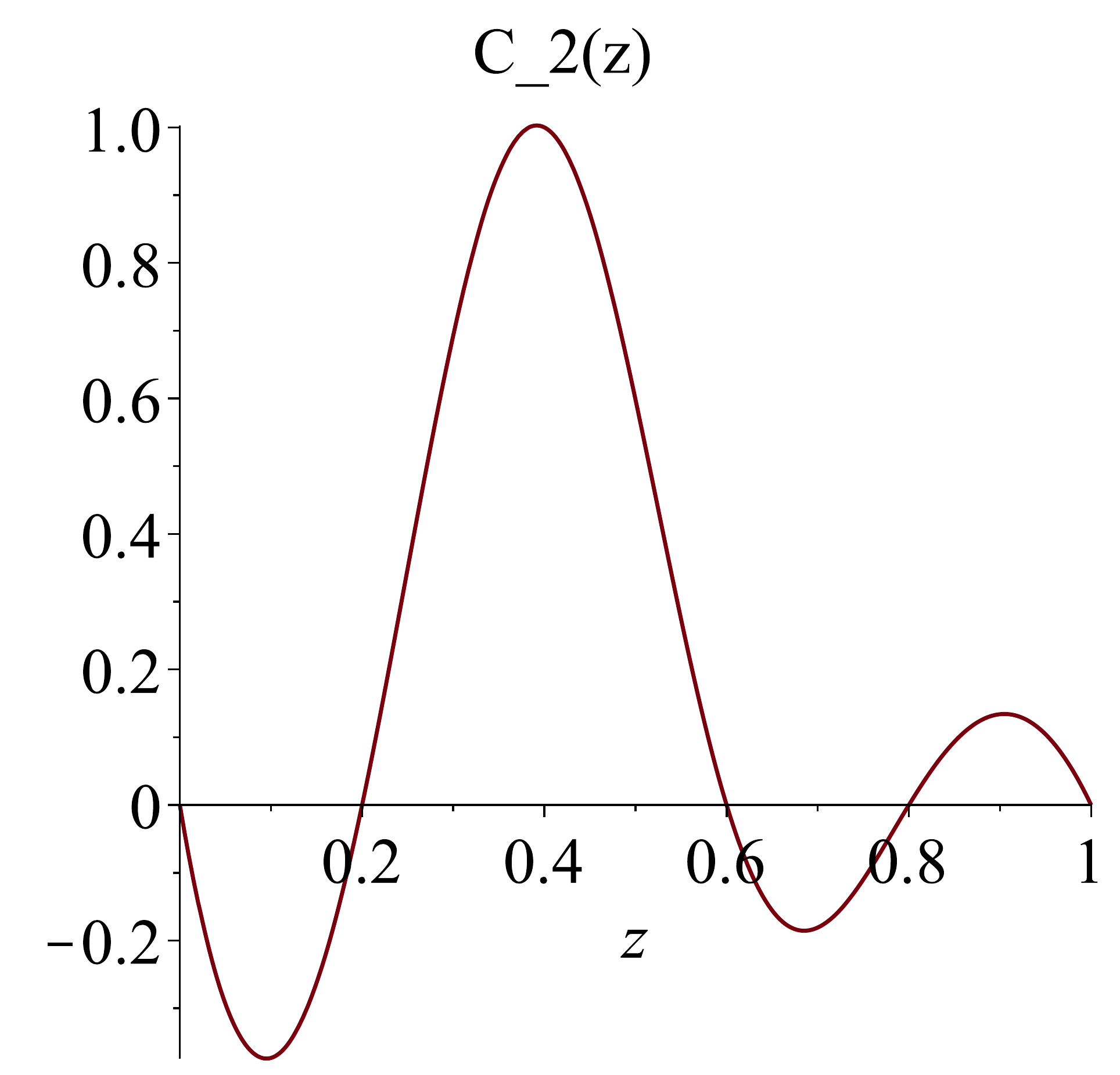}\\
  \includegraphics[width=0.3\textwidth]{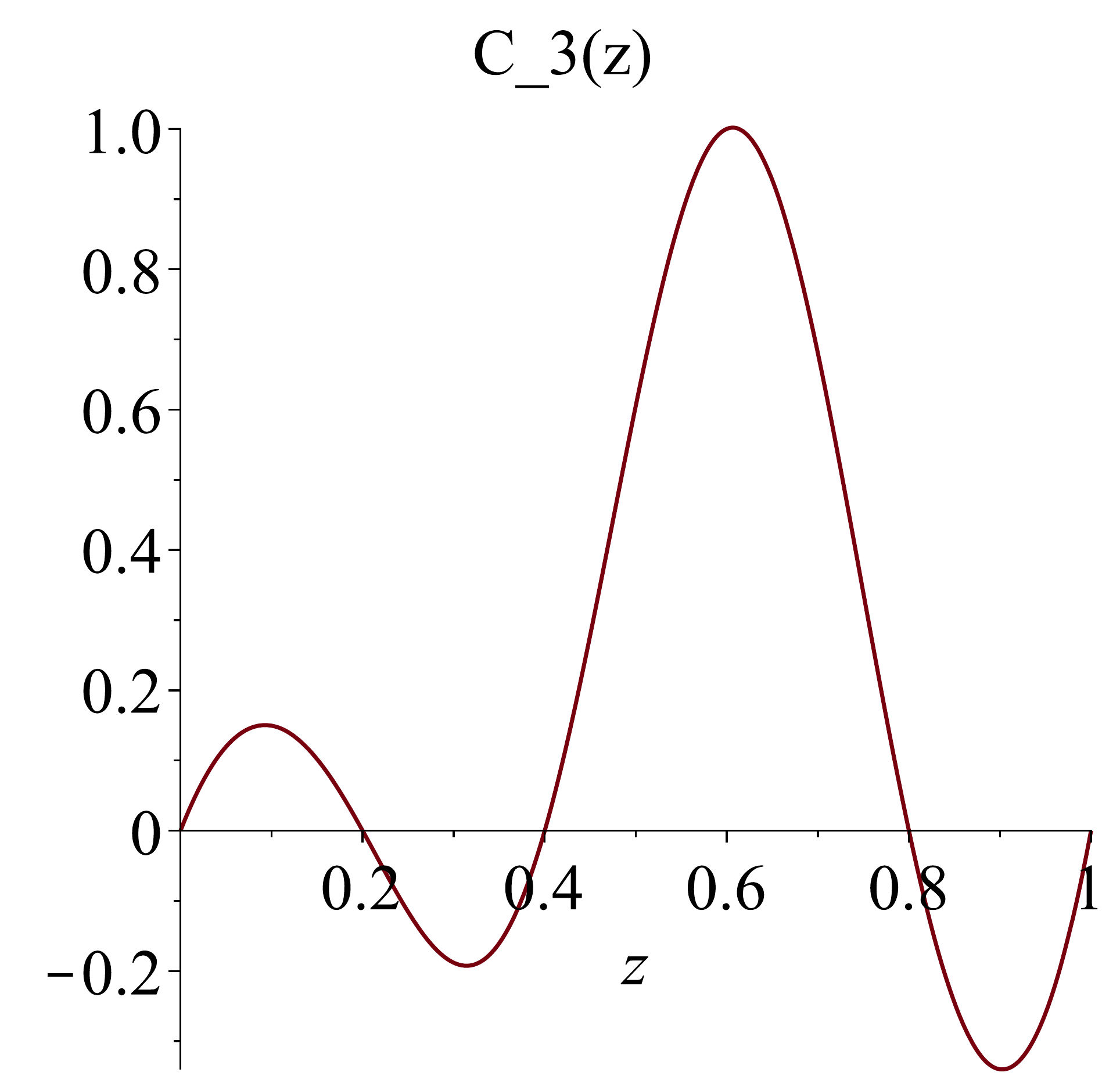}
  \includegraphics[width=0.3\textwidth]{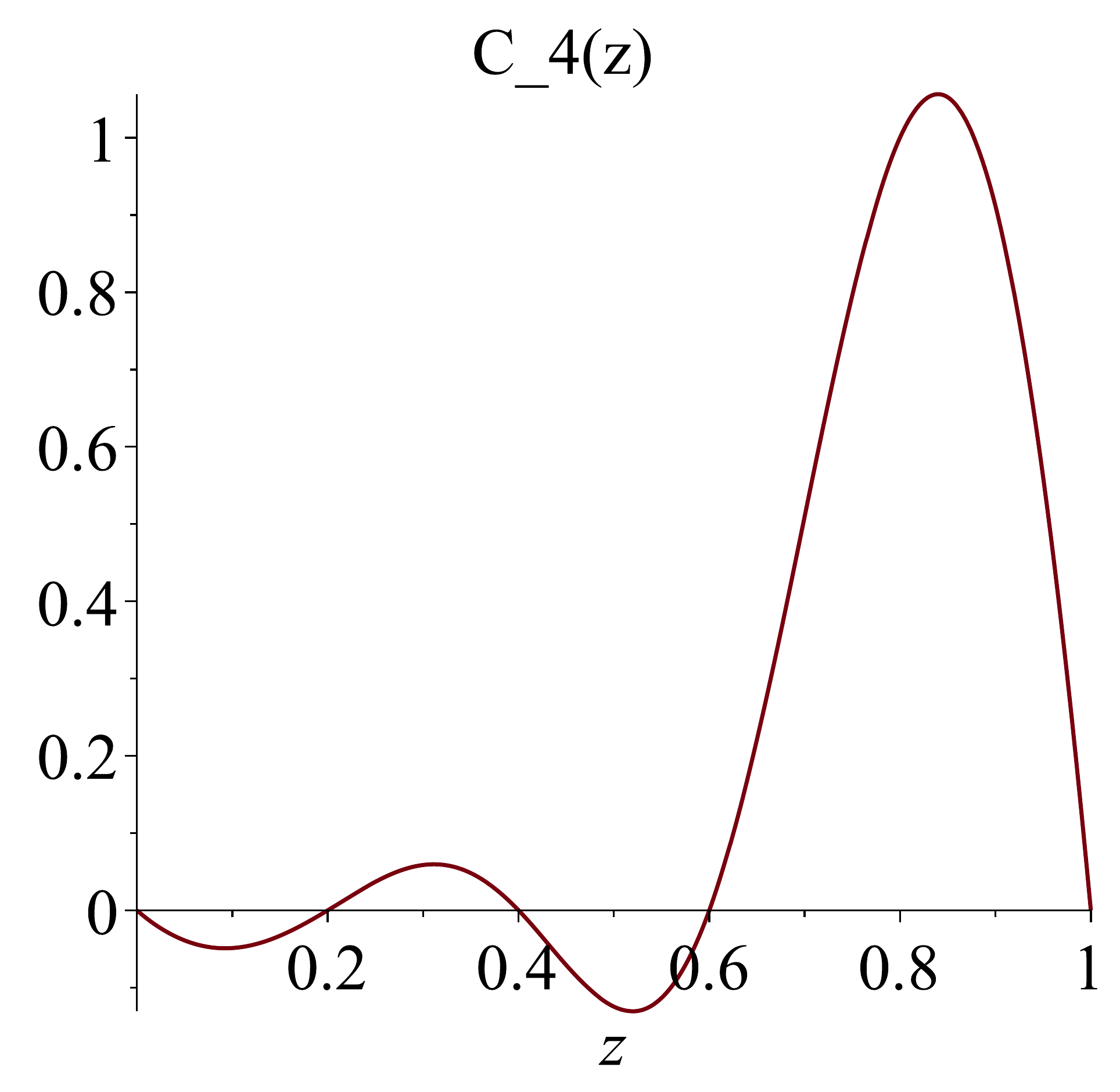}
  \includegraphics[width=0.3\textwidth]{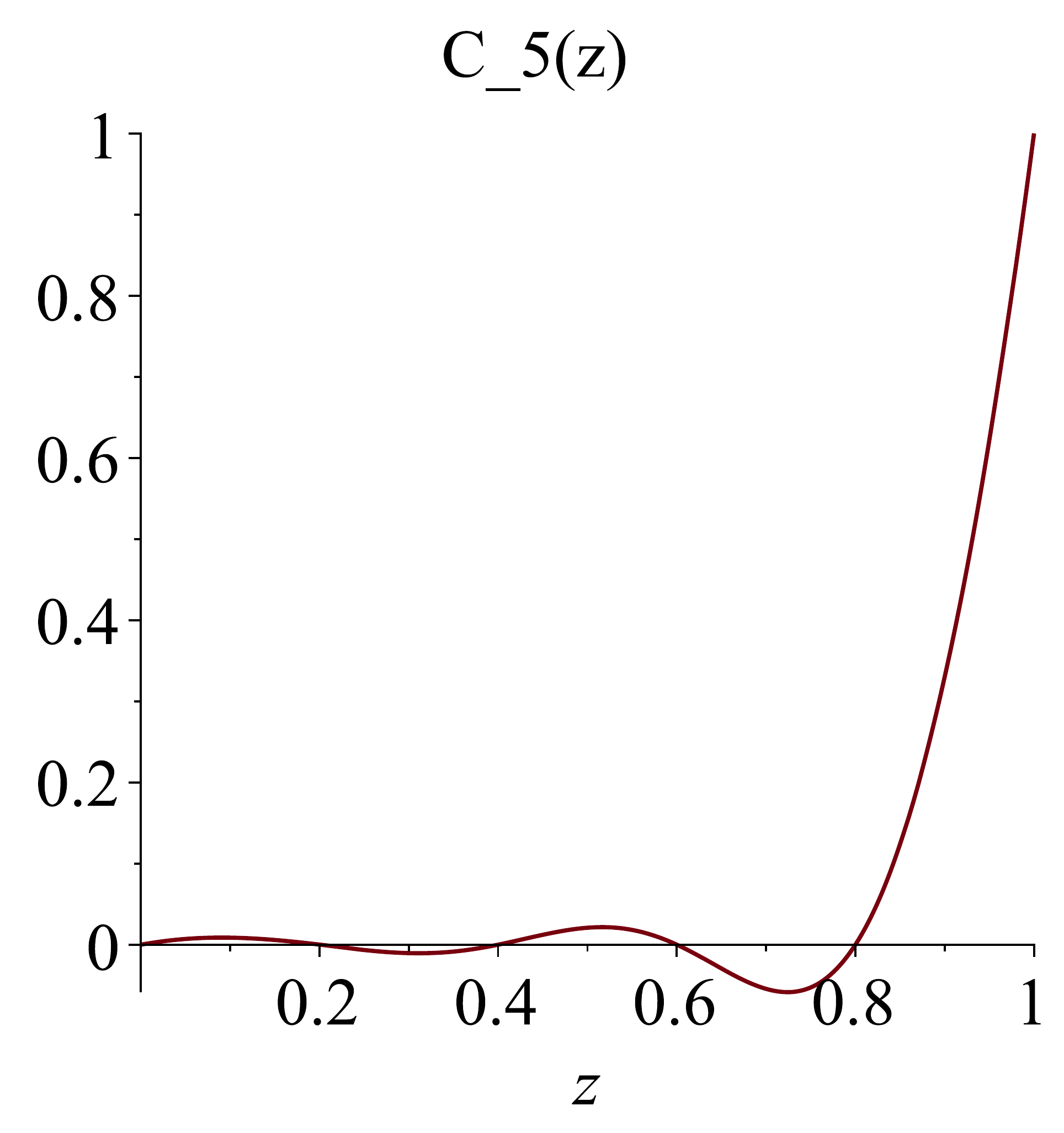}
  \caption{Graphs of coefficients of the optimal interpolation formulas (\ref{(1)}) in the case $m=3$ and $N=5$.}\label{Fig7}
\end{figure}

\begin{figure}
  \includegraphics[width=0.3\textwidth]{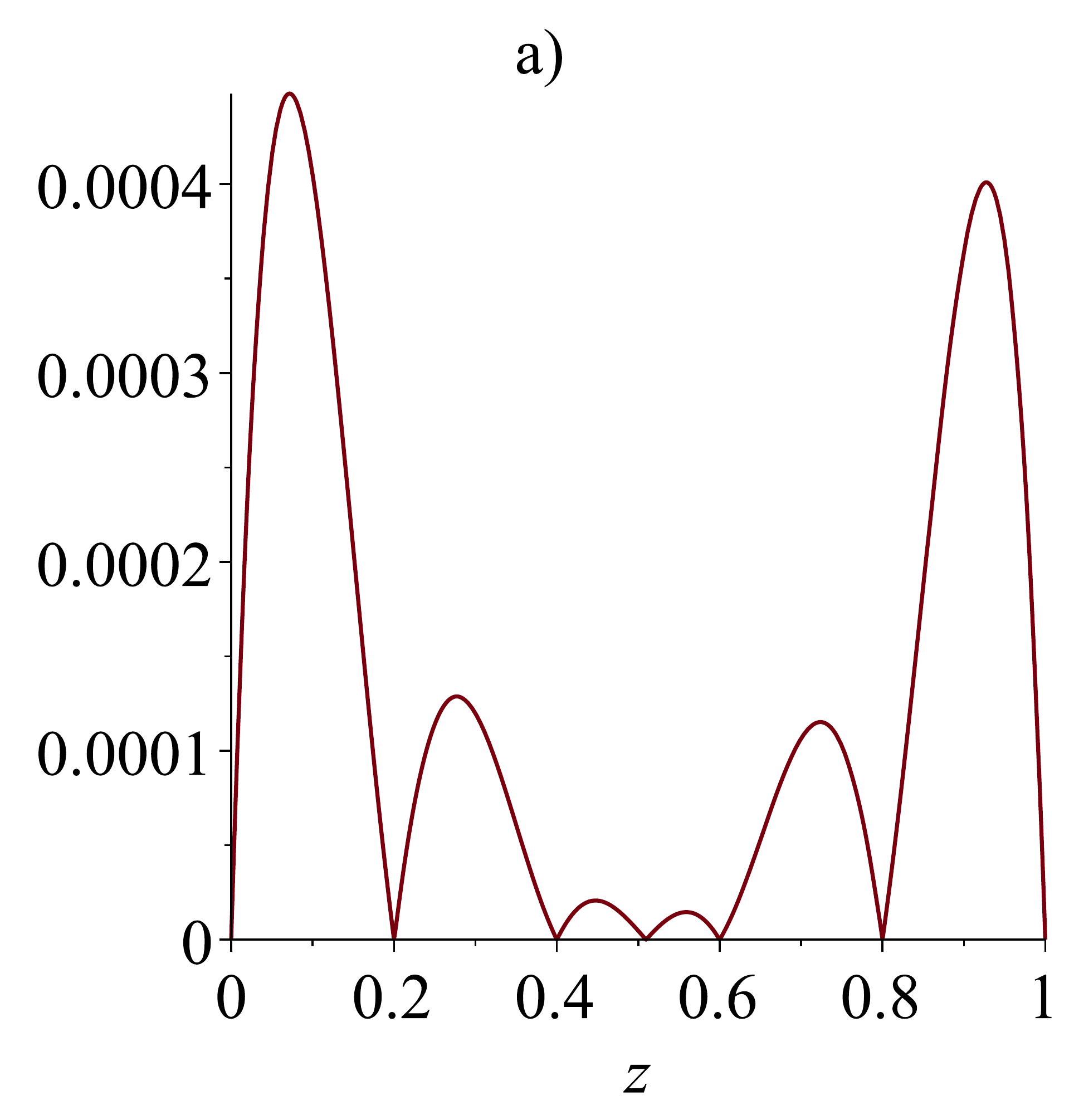}
  \includegraphics[width=0.3\textwidth]{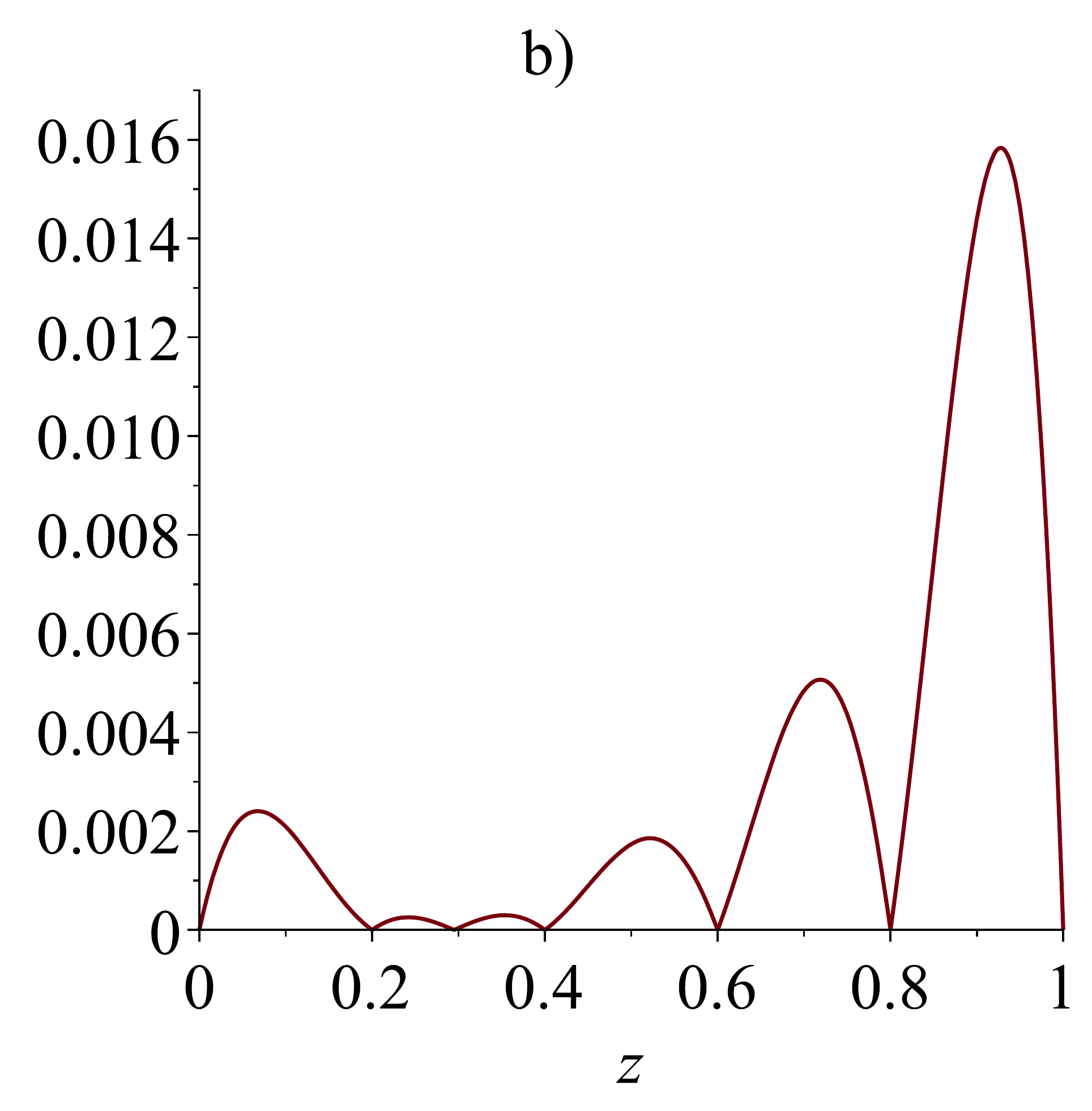}
  \includegraphics[width=0.3\textwidth]{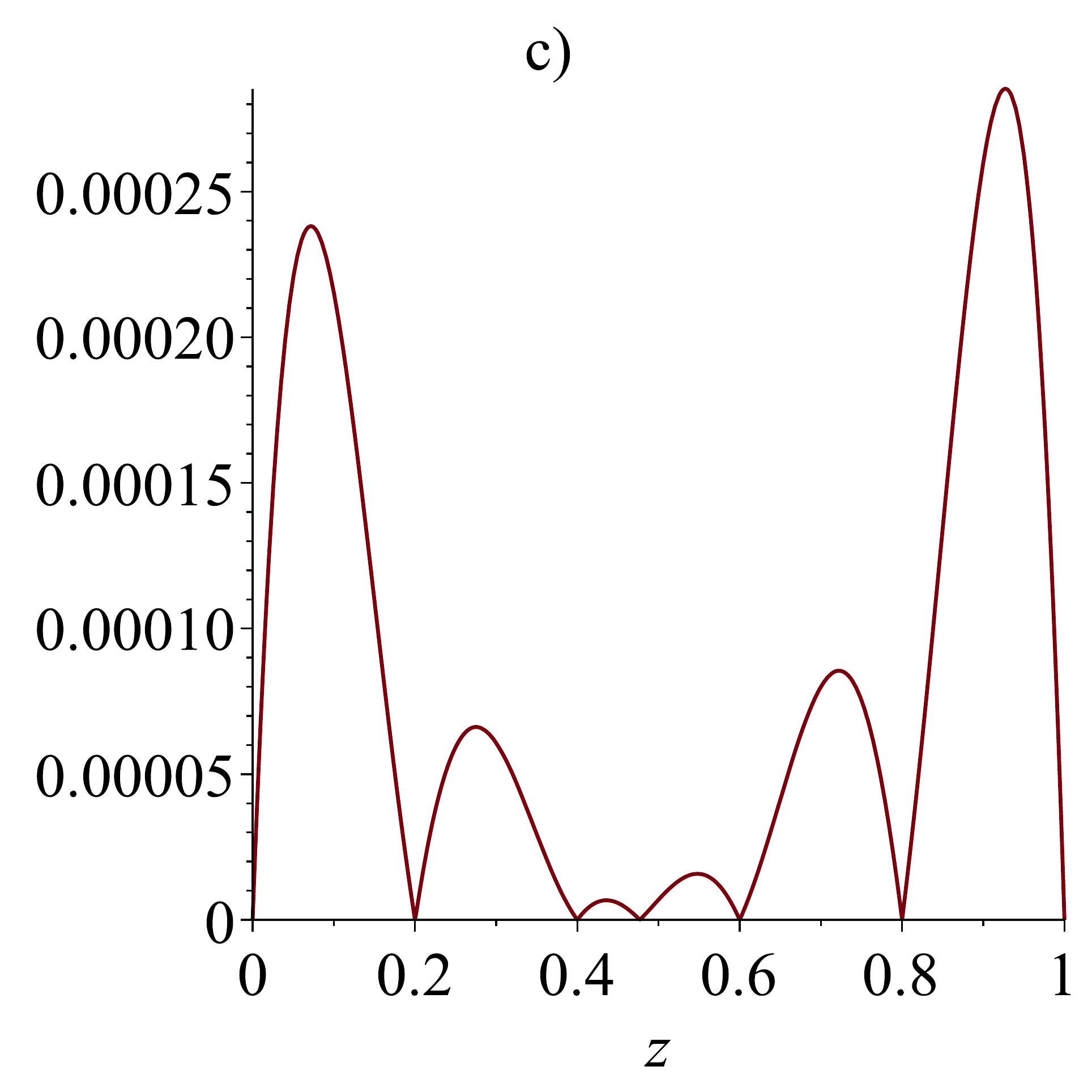}
  \caption{Graphs of absolute errors for $m=3$ and $N=5$: (\textbf{a}) $|\varphi_1(z)-P_{\varphi_1}(z)|$, (\textbf{b}) $|\varphi_2(z)-P_{\varphi_2}(z)|$, (\textbf{c}) $|\varphi_3(z)-P_{\varphi_3}(z)|$.}\label{Fig8}
\end{figure}

\begin{figure}
  \includegraphics[width=0.3\textwidth]{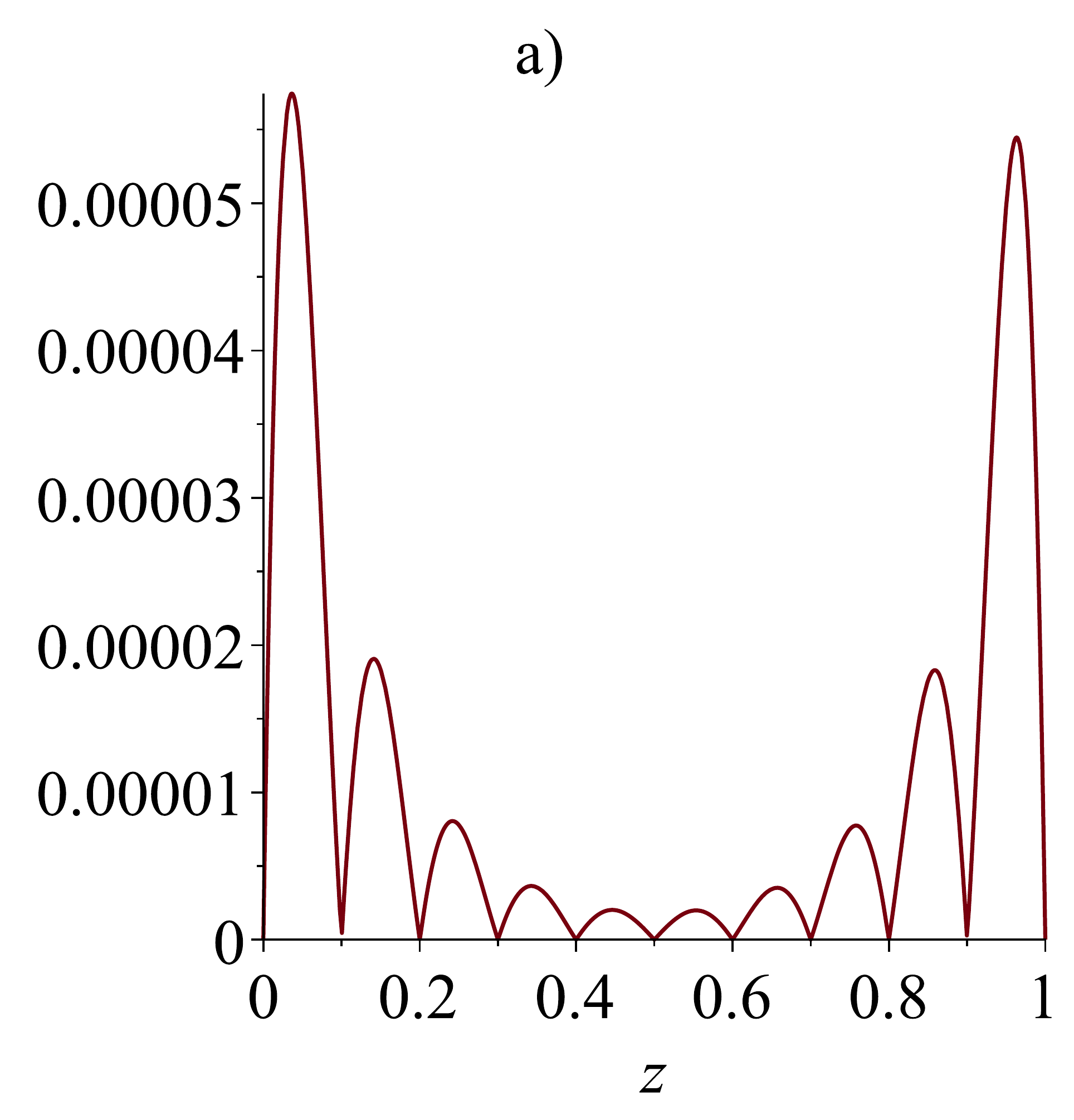}
  \includegraphics[width=0.3\textwidth]{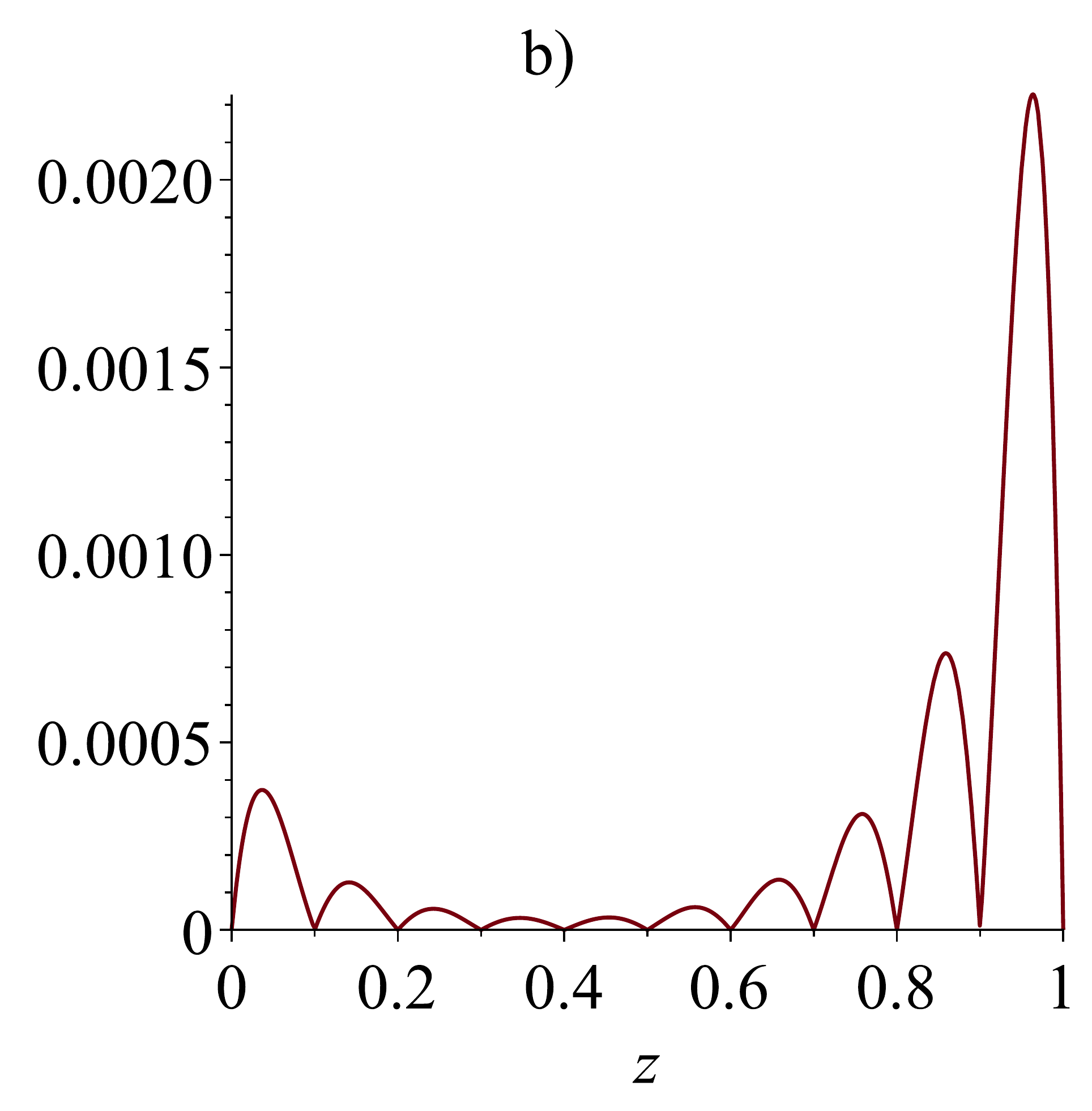}
  \includegraphics[width=0.3\textwidth]{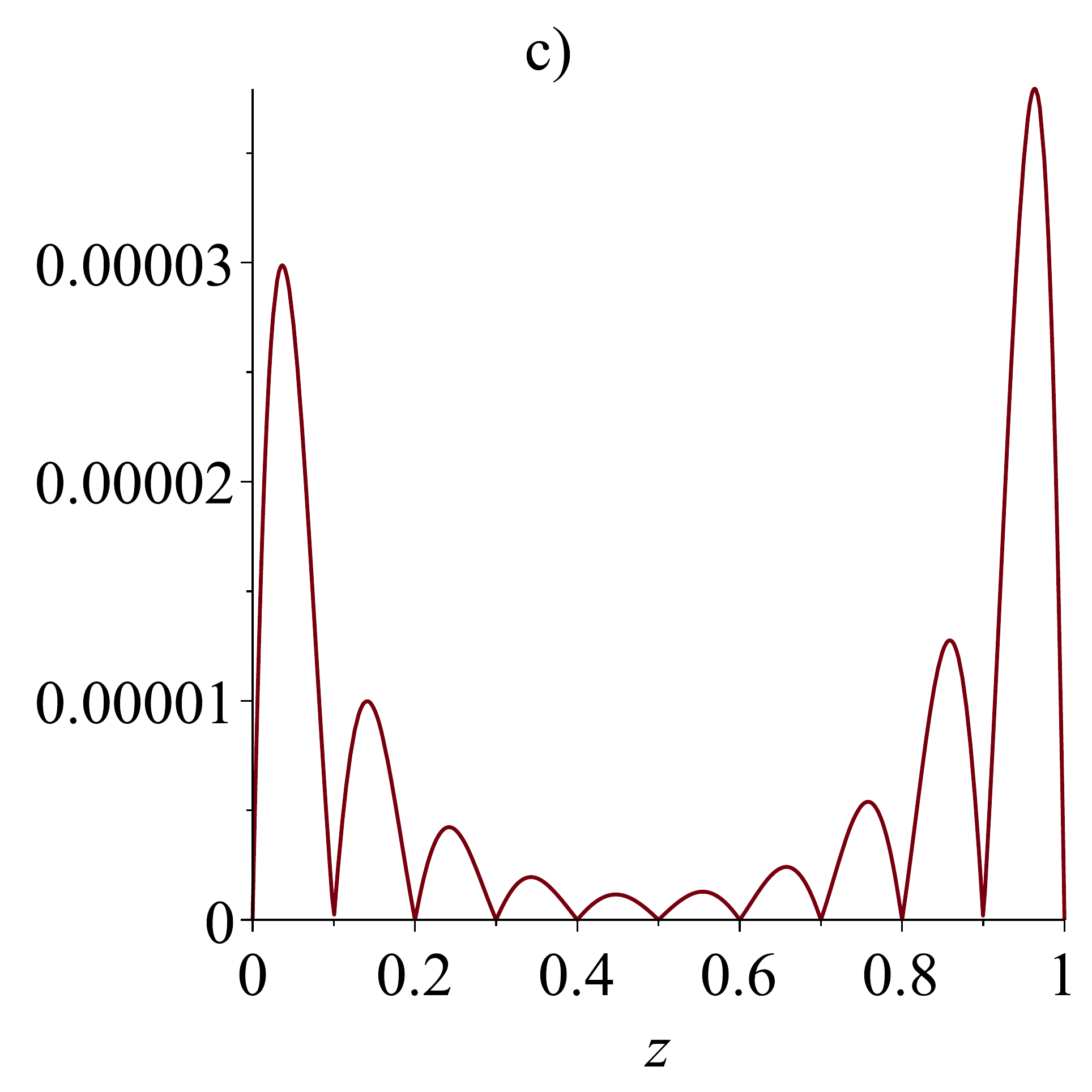}
  \caption{Graphs of absolute errors for $m=3$ and $N=10$: (\textbf{a}) $|\varphi_1(z)-P_{\varphi_1}(z)|$, (\textbf{b}) $|\varphi_2(z)-P_{\varphi_2}(z)|$, (\textbf{c}) $|\varphi_3(z)-P_{\varphi_3}(z)|$.}\label{Fig9}
\end{figure}

\end{document}